\documentclass[reqno, 11pt]{amsart}
\usepackage[margin=1in]{geometry}
\usepackage{amsthm, amsmath, amsfonts, amssymb}
\usepackage{enumerate, mathabx, tikz}
\usepackage[hyphens]{url}
\usepackage{hyperref}
\usepackage[hyphenbreaks]{breakurl}

\usepackage{color}

\newcommand{\R}[0]{\mathbb{R}}

\newcommand{\T}[0]{\mathbb{T}}

\newcommand{\ta}[0]{\theta}
\newcommand{\bs}[0]{\setminus}

\newcommand{\vep}[0]{\varepsilon}
\newcommand{\supp}[0]{\operatorname{supp}}
\newcommand{\lsm}[0]{\lesssim}

\newcommand{\om}[0]{\omega}
\newcommand{\vp}[0]{\varphi}

\newcommand{\wh}[1]{\widehat{#1}}
\newcommand{\wc}[1]{\widecheck{#1}}
\newcommand{\mc}[1]{\mathcal{#1}}

\newcommand{\ov}[1]{\overline{#1}}
\newcommand{\wt}[1]{\widetilde{#1}}
\newcommand{\st}[1]{\substack{#1}}
\newcommand{\mb}[1]{\mathbf{#1}}

\newcommand{\nms}[1]{\| #1 \|}

\newtheorem{thm}{Theorem}
\newtheorem{lemma}[thm]{Lemma}
\newtheorem{prop}[thm]{Proposition}
\newtheorem{cor}[thm]{Corollary}

\theoremstyle{remark}

\title{Mixed norm $l^2$ decoupling for paraboloids}
\author{Shival Dasu}
\address{Department of Mathematics, Indiana University Bloomington, Bloomington, IN 47405, USA}
\email{sdasu@iu.edu}

\author{Hongki Jung}
\address{Department of Mathematics, Indiana University Bloomington, Bloomington, IN 47405, USA}
\email{jung11@iu.edu}

\author{Zane Kun Li}
\address{Department of Mathematics, University of Wisconsin-Madison, Madison, WI 53706, USA}
\email{zkli@wisc.edu}

\author{Jos\'e Madrid}
\address{Section de Math\'ematiques, Universit\'e de Gen\`eve,  Geneva, Switzerland}
\email{Jose.Madrid@unige.ch}

\begin{document}
\begin{abstract}
We prove the sharp mixed norm $(l^2, L^{q}_{t}L^{r}_{x})$ decoupling estimate for the paraboloid in $d + 1$ dimensions.
\end{abstract}
\maketitle
\section{Introduction}
For $\delta \in (0, 1)$ and $\om \in \R^d$, 
define the parallelpiped
\begin{align*}
\ta_{\om} := \{(\xi_1, \ldots, \xi_{d}, \eta) \in \R^{d + 1}: |\xi_i - \om_i| \leq \delta \text{ for } i = 1, 2, \ldots, d,\, |\eta - 2\om \cdot (\xi -\om) - |\om|^2| \leq 2d\delta^2\}.
\end{align*}
Sometimes to emphasize the dependence of $\ta_{\om}$ on $\delta$, we will
write $\ta_{\om, \delta}$ instead.
Note that $\ta_{\om}$ contains the $\delta^{2}$ neighborhood of the piece of paraboloid
above $\prod_{i = 1}^{d}[\om_{i}, \om_{i} + \delta)$ given by 
\begin{align*}
\{(\xi_1, \ldots, \xi_d, \eta) \in \R^{d + 1}: \xi_{i} \in [\om_{i}, \om_{i} + \delta) \text{ for } i = 1, 2, \ldots, d,\, |\eta - |\xi|^2| \leq \delta^2\}
\end{align*}
since if $|\eta - |\xi|^2| \leq \delta^2$, then $|\eta - 2\om \cdot (\xi - \om) - |\om|^2| \leq |\eta - |\xi|^2| + |\xi - \om|^{2}  \leq 2d\delta^2$.
The parallelpiped $\ta_{\om}$ is comparable to a $\delta \times \cdots \times \delta \times \delta^{2}$ rectangular box and is essentially the $O(\delta^{2})$ neighborhood of the piece of paraboloid living above $B(\om, \delta)$.

Next, for a $\ta \subset \R^{d+1}$, let $P_{\ta}f := (\wh{f}1_{\ta})^{\vee}$. Let $\Xi$ be a $\delta$-separated subset of $[-1, 1]^{d}$.
For $1 \leq q, r \leq \infty$, let $D_{q, r}(\delta, \Xi)$ be the best constant such that
\begin{align*}
\nms{\sum_{\om \in \Xi}P_{\ta_{\om}}f}_{L^{q}_{t}L^{r}_{x}(\R^d \times \R)} \leq D_{q, r}(\delta, \Xi)(\sum_{\om \in \Xi}\nms{P_{\ta_{\om}}f}_{L^{q}_{t}L^{r}_{x}(\R^d \times \R)}^{2})^{1/2}
\end{align*}
for all $f$ with Fourier transform supported in $\bigcup_{\om \in \Xi}\ta_{\om}$. Finally, we define
\begin{align}\label{sharpdecoupling}
D_{q, r}(\delta) := \sup_{\Xi}D_{q, r}(\delta, \Xi)
\end{align}
where the supremum is taken over all $\delta$-separated subsets $\Xi$ of $[-1, 1]^d$.
The way we defined $\ta_{\om}$ and the decoupling constant here was inspired
by a formulation due to Tao \cite{247B}.
By Plancherel, $D_{2, 2}(\delta) \lsm 1$.
Bourgain and Demeter's sharp $(l^{2}, L^{p}_{x, t})$ paraboloid decoupling theorem \cite{BD14} implies that
when $2 \leq p \leq \frac{2(d + 2)}{d}$, one has $D_{p, p}(\delta) \lsm_{p, \vep} \delta^{-\vep}$.
After \cite{BD14}, it was asked if there were mixed norm decoupling theorems,
see for example \cite{Barron, BBGL, YZ, MO}.
To this end, we give the following sharp mixed norm $(l^{2}, L^{q}_{t}L^{r}_{x})$
decoupling for the paraboloid:
\begin{thm}\label{main}
For every $\vep > 0$,  we have
\begin{align}\label{mainbd}
D_{q, r}(\delta) \lsm_{q, r, \vep} \delta^{-\vep}
\end{align}
for all $\delta \in (0, 1)$ if and only if
\begin{align}\label{qrregion}
2 \leq q \leq \infty, \quad 2  \leq r \leq \frac{2(d + 2)}{d}, \quad \text{and}\quad \frac{2}{q} + \frac{d}{r}  \geq \frac{d}{2}.
\end{align} 
\end{thm}
The region given in \eqref{qrregion} is illustrated in Figure \ref{mixedfig} below. 
\begin{figure}
\begin{tikzpicture}[scale=5]
\draw[->][thick] (0,0)-- (0.6,0);
\draw[->][thick] (0,0) -- (0,0.6);
\node[below] at (0.6,0) {$\frac{1}{r}$};
\node[left] at  (0,0.65) {$\frac{1}{q}$};
\draw[thick] (1/2,-0.02) node[below] {$\frac{1}{2}$} -- (1/2,0.02);
\draw[thick] (-0.02, 1/4) node[left] {$\frac{1}{4}$} -- (0.02, 1/4);
\draw[thick] (-0.02, 1/2) node[left] {$\frac{1}{2}$} -- (0.02, 1/2);
\draw[very thick] (1/2, 1/2) circle[radius=0.05ex];
\draw[very thick] (1/2, 0) circle[radius=0.05ex];
\draw[very thick] (1/6, 1/2) circle[radius=0.05ex];
\node[above] at (1/6, 1/2) {\small $(\frac{1}{6}, \frac{1}{2})$};
\node[above left] at (1/2 + 0.1,1/2) {\small $( \frac{1}{2}, \frac{1}{2}) $ };
\draw[very thick] (1/6, 1/6) circle[radius=0.05ex];
\node[below left] at (1/6 + 0.1,1/6) {\small $( \frac{1}{6}, \frac{1}{6}) $ };
\draw[thick] (1/2, 0) -- (1/6, 1/6);
\draw[dashed] (1/6, 1/6) -- (0, 1/4);
\draw[thick] (1/6, 1/6) -- (1/6, 1/2);
\draw[fill=gray, opacity = 0.3] (1/6, 1/6) -- (1/6, 1/2) -- (1/2, 1/2) -- (1/2, 0);
\draw[thick] (1/6, 1/2) -- (1/2, 1/2);
\draw[dashed] (0, 1/2) -- (1/6, 1/2);
\draw[thick] (1/2, 1/2) -- (1/2, 0);
\end{tikzpicture}
\quad\quad\quad
\begin{tikzpicture}[scale=5]
\draw[->][thick] (0,0)-- (0.6,0);
\draw[->][thick] (0,0) -- (0,0.6);
\node[below] at (0.6,0) {$\frac{1}{r}$};
\node[left] at  (0,0.65) {$\frac{1}{q}$};
\draw[thick] (1/2,-0.02) node[below] {$\frac{1}{2}$} -- (1/2,0.02);
\draw[thick] (-0.02, 1/2) node[left] {$\frac{1}{2}$} -- (0.02, 1/2);
\draw[very thick] (1/2, 1/2) circle[radius=0.05ex];
\draw[very thick] (1/2, 0) circle[radius=0.05ex];
\draw[very thick] (1/4, 1/2) circle[radius=0.05ex];
\node[above] at (1/4, 1/2) {\small $(\frac{1}{4}, \frac{1}{2})$};
\node[above left] at (1/2 + 0.1,1/2) {\small $( \frac{1}{2}, \frac{1}{2}) $ };
\draw[very thick] (1/4, 1/4) circle[radius=0.05ex];
\node[below left] at (1/4,1/4) {\small $( \frac{1}{4}, \frac{1}{4}) $ };
\draw[thick] (1/2, 0) -- (1/4, 1/4);
\draw[dashed] (1/4, 1/4) -- (0, 1/2);
\draw[thick] (1/4, 1/4) -- (1/4, 1/2);
\draw[fill=gray, opacity = 0.3] (1/4, 1/4) -- (1/4, 1/2) -- (1/2, 1/2) -- (1/2, 0);
\draw[dashed] (0, 1/2) -- (1/4, 1/2);
\draw[thick] (1/4, 1/2) -- (1/2, 1/2);
\draw[thick] (1/2, 1/2) -- (1/2, 0);
\end{tikzpicture}
\quad\quad\quad
\begin{tikzpicture}[scale=5]
\draw[->][thick] (0,0)-- (0.6,0);
\draw[->][thick] (0,0) -- (0,0.6);
\node[below] at (0.6,0) {$\frac{1}{r}$};
\node[left] at  (0,0.65) {$\frac{1}{q}$};
\draw[thick] (1/2,-0.02) node[below] {$\frac{1}{2}$} -- (1/2,0.02);
\draw[thick] (-0.02, 1/2) node[left] {$\frac{1}{2}$} -- (0.02, 1/2);
\draw[very thick] (1/2, 1/2) circle[radius=0.05ex];
\draw[very thick] (1/2, 0) circle[radius=0.05ex];
\draw[very thick] (3/10, 1/2) circle[radius=0.05ex];
\node[above left] at (3/10 + 0.1, 1/2) {\small $(\frac{d}{2(d + 2)}, \frac{1}{2})$};
\node[above left] at (1/2 + 0.1,1/2) {\small $( \frac{1}{2}, \frac{1}{2}) $ };
\draw[very thick] (3/10, 3/10) circle[radius=0.05ex];
\node[below left] at (3/10 + 0.15,3/10) {\small $( \frac{d}{2(d + 2)}, \frac{d}{2(d + 2)}) $ };
\draw[thick] (1/2, 0) -- (3/10, 3/10);
\draw[dashed] (3/10, 3/10) -- (1/6, 1/2);
\draw[thick] (3/10, 3/10) -- (3/10, 1/2);
\draw[fill=gray, opacity = 0.3] (3/10, 3/10) -- (3/10, 1/2) -- (1/2, 1/2) -- (1/2, 0);
\draw[dashed] (1/6, 1/2) -- (3/10, 1/2);
\draw[thick] (3/10, 1/2) -- (1/2, 1/2);
\draw[thick] (1/2, 1/2) -- (1/2, 0);
\end{tikzpicture}
\caption{$d = 1$ (left), $d = 2$ (center), and $d \geq 3$ (right)}\label{mixedfig}
\end{figure}
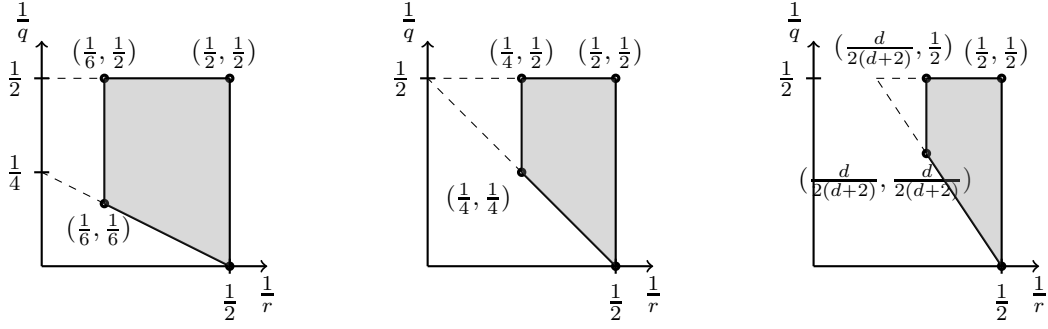
Outside the shaded region in Figure \ref{mixedfig},
interpolation with trivial bounds and supercritical $(l^{2}, L^{p}_{x, t})$ paraboloid decoupling gives the sharp estimate.
More precisely, we have:
\begin{prop}\label{full}
For any $q, r \geq 2$,
\begin{align}\label{fulleq}
D_{q, r}(\delta) \lsm_{q, r, \vep}
\begin{cases}
\delta^{-(\frac{d}{2} - \frac{d}{r} - \frac{2}{r}) - \vep} & \text{ if }\quad  q \leq r \text{ and }\, r \geq \frac{2(d + 2)}{d}\\
\delta^{-(\frac{d}{2} - \frac{d}{r} - \frac{2}{q}) - \vep} & \text{ if }\quad  q \geq r \text{ and }\, \frac{2}{q} + \frac{d}{r} \leq \frac{d}{2}
\end{cases}
\end{align}
and this estimate is sharp up to a $\delta^{-\vep}$ loss.
\end{prop}

One corollary of Theorem \ref{main} is that if
$(q, r)$ lie in \eqref{qrregion}, then
\begin{align}\label{discres}
\nms{\sum_{n_1, \ldots, n_d = 1}^{N}a_{n}e^{2\pi i (n \cdot x + |n|^2 t)}}_{L^{q}_{t}L^{r}_{x}(\T^{d} \times [0, 1])} \lsm_{q, r, \vep} N^{\vep}(\sum_{n_1, \ldots, n_d = 1}^{N}|a_n|^{2})^{1/2}.
\end{align}
The case of $(q, r) = (2, \frac{2d}{d - 2})$ in \eqref{discres} corresponds to where the Strichartz line intersects the line $\frac{1}{q} = \frac{1}{2}$ is of interest in PDE,
see for example \cite[Section 6.2]{KNS}, though our result is only able to give:
\begin{align}\label{torusimprove}
\nms{\sum_{n_1, \ldots, n_d = 1}^{N}a_{n}e^{2\pi i (n \cdot x + |n|^2 t)}}_{L^{2}_{t}L^{\frac{2d}{d - 2}}_{x}(\T^{d} \times [0, 1])} \lsm_{\vep} N^{2/d+\vep}(\sum_{n_1, \ldots, n_d = 1}^{N}|a_n|^{2})^{1/2}.
\end{align}
An application of \cite[Theorem 1]{BGT} in the case of the torus gives \eqref{torusimprove}
with a loss of $N^{1/2}$ instead of $N^{2/d}$ and thus \eqref{torusimprove} represents an improvement in this case for $d \geq 5$.
Note however, Theorem \ref{main} does not discount the possibility that the weaker discrete restriction estimate \eqref{discres}
is valid on the larger region bounded by the Strichartz line and the lines $\frac{1}{q} = \frac{1}{2}$ and $\frac{1}{r} = \frac{1}{2}$.
The example in Section \ref{sec:r2d2d} shows the necessity of $r \leq \frac{2(d + 2)}{d}$ and is built out of wavepackets that have constructive
interference on temporally disjoint intervals.
This was inspired from a ``wavepacket tuning" idea mentioned by Larry Guth in a talk\footnote{\url{https://www.youtube.com/watch?v=h_psKQV-pU4}} given at the HIM Dual Trimester
Program on Harmonic Analysis and Analytic Number Theory in 2021.

The proof that \eqref{qrregion} is necessary for \eqref{mainbd}
follows from four different lower bounds (see Sections \ref{sec:line}-\ref{sec:r2d2d}). 
For the sufficiency of \eqref{qrregion} for \eqref{mainbd}, 
we first reduce to a more smoothed version of decoupling.
We then use an interpolation estimate for decoupling
combined with a reverse H\"{o}lder in time (see Proposition \ref{decrease})
and then apply both results to the known sharp $(l^{2}, L_{x, t}^{2(d + 2)/d})$ paraboloid
decoupling estimate due to Bourgain and Demeter \cite{BD14}. 

\subsection*{Notation}
We will let $e(t) := e^{2\pi i t}$.
All our implied constants are allowed to depend on the spatial dimension $d$.
Given two nonnegative expressions $X$ and $Y$, we write
$X \lsm Y$ if $X \leq CY$ for some constant $C$ that is allowed to depend on $d$. If $C$ depends on some additional parameter $A$,
then we write $X \lsm_{A} Y$. We write $X \sim Y$ if $X \lsm Y$ and $Y \lsm X$.
Let $c \theta$ be defined to be a region with the same center
as $\theta$ but dilated by $c$.
We will let $\phi$ be an even Schwartz function such that
\begin{align}\label{schwartzphi}
1_{[-1, 1]}(t) \leq \phi(t) \lsm_{E} (1 + |t|)^{-E}
\end{align}
and $\supp(\wh{\phi}) \subset [-1, 1]$
where we are free to choose $E$.
An explicit example of such a $\phi$ can be constructed from a modification
of the function constructed  in \cite[Lemma 2.2.8]{thesis} which
gave an even function $\geq 1_{[-1/2, 1/2]}$ and whose Fourier support
is contained in $[-1/2, 1/2]$.
We define the mixed norms:
\begin{align*}
\nms{f}_{L^{q}_{t}L^{r}_{x}(\R^d \times \R)} &:= (\int_{\R}(\int_{\R^d}|f(x, t)|^{r}\, dx)^{q/r}\, dt)^{1/q}\\
\nms{f}_{L^{q}_{t}L^{r}_{x}(\R^d \times I)} &:= \nms{f1_{I}}_{L^{q}_{t}L^{r}_{x}(\R^d \times \R)} = (\int_{I}(\int_{\R^d}|f(x, t)|^{r}\, dx)^{q/r}\, dt)^{1/q}\\
\nms{f}_{L^{q}_{t}L^{r}_{x}(\R^d \times \phi)} &:= \nms{f\phi}_{L^{q}_{t}L^{r}_{x}(\R^d \times \R)}= (\int_{\R}(\int_{\R^d}|f(x, t)|^{r}\, dx)^{q/r}\phi(t)^{q}\, dt)^{1/q}
\end{align*}
where $\phi(t) : \R \rightarrow \R_{\geq 0}$ and $I$ is a time interval.
For an interval $I$ centered at $c$ of length $R$, we let $\phi_{I}(t) := \phi(\frac{t - c}{R})$.
Finally, we define
\begin{align*}
\nms{f}_{L^{q}_{\#, t}L^{r}_{x}(\R^d \times I)} &:= |I|^{-1/q}\nms{f1_I}_{L^{q}_{t}L^{r}_{x}(\R^d \times \R)} = (\frac{1}{|I|}\int_{I}(\int_{\R^d}|f(x, t)|^{r}\, dx)^{q/r}\, dt)^{1/q}\\
\nms{f}_{L^{q}_{\#, t}L^{r}_{x}(\R^d \times \phi_{I})} &:= |I|^{-1/q}\nms{f\phi_{I}}_{L^{q}_{t}L^{r}_{x}(\R^d \times \R)} = (\frac{1}{|I|}\int_{\R}(\int_{\R^d}|f(x, t)|^{r}\, dx)^{q/r}\phi_{I}(t)^{q}\, dt)^{1/q}.
\end{align*}

\subsection*{Acknowledgements}
ZL is supported by NSF grants DMS-2037851 and DMS-2311174.
The authors
would like to thank Ciprian Demeter, Larry Guth, and Terence Tao
for numerous helpful discussions and encouragement.
The authors would also like to thank Nikolay Tzvetkov for letting us know about the comparison of our result \eqref{torusimprove} to \cite[Theorem 1]{BGT} in the case of the torus.
\section{Proof of necessity of \eqref{qrregion} for Theorem \ref{main}}
To prove the necessity of \eqref{qrregion}, we will give four different
examples. 
In this section, for $\delta \in (0, 1)$, we choose $\Xi$ to be a maximal $\delta$-net of $[-1, 1]^d$,
that is, $\Xi$ is a collection of points in $[-1, 1]^d$ that are separated
from each other by a distance of at least $\delta$ and which is maximal
with respect to set inclusion.
A volume counting argument implies that $\#\Xi \sim \delta^{-d}$.
In our examples below, for brevity, we will abbreviate $P_{\ta_{\om}}f$ with $f_{\ta_{\om}}$.
 
We will let $\mb{D}_{q, r}(\delta) := D_{q, r}(\delta, \Xi)$. 
We will show that for all $1 \leq q, r \leq \infty$ we have
\begin{align}\label{lowerboundsum}
\mb{D}_{q, r}(\delta) \gtrsim_{q, r} \max(\delta^{-(\frac{d}{2} - \frac{d}{r} - \frac{2}{q})}, \delta^{-d(\frac{1}{r} - \frac{1}{2})}, \delta^{-d(\frac{1}{q} - \frac{1}{2})}, \delta^{-(\frac{d}{2} - \frac{d + 2}{r})})
\end{align}
for all $\delta$ sufficiently small depending on $d$.
The expressions on the right hand side will be shown in Sections \ref{sec:line}, \ref{sec:r2}, \ref{sec:q2} and \ref{sec:r2d2d}, respectively.
Once we have \eqref{lowerboundsum}, since $D_{q, r}(\delta) \geq \mb{D}_{q, r}(\delta)$, it follows that to have \eqref{mainbd}, we must
have the conditions in \eqref{qrregion} thereby proving
necessity of these conditions.

\subsection{Dual tubes and Schwartz functions}
For notational convenience, let
\begin{align*}
\ta_{\delta} := \ta_{0, \delta} = \{(\xi, \eta) \in \R^{d + 1}: |\xi_i| \leq \delta \text{ for } i = 1, 2, \ldots, d,\, |\eta| \leq 2d\delta^2\}
\end{align*}
and define its dual parallelpiped
\begin{align*}
T_{\delta} := \{(x, t) \in \R^{d + 1}: |x_i| \leq \delta^{-1} \text{ for } i = 1, 2, \ldots, d,\, |t| \leq (2d)^{-1}\delta^{-2}\}.
\end{align*}
Define the $(d + 1) \times (d + 1)$ matrix
\begin{align*}
M_{\om} := \left[\begin{array}{c|c}
I_{d} & \mathbf{0}\\
\hline
2\om & 1
\end{array}\right]
\end{align*}
where here $I_{d}$ is the $d \times d$ identity matrix, $\mb{0}$ is a $d\times 1$ column
vector of $0$'s, and $2\om$ is written as a $1 \times d$ row vector. 
Note that $M_{\om}^{-1} = (\begin{smallmatrix} I_d &\mathbf{0}\\-2\om & 1\end{smallmatrix})$.
Observe that
\begin{align}\label{taqtad}
\ta_{\om} = (\om, |\om|^{2}) + M_{\om}\ta_{\delta}.
\end{align}
With $\phi$ as in \eqref{schwartzphi}, define
\begin{align*}
\psi_{\ta_{\delta}}(\xi, \eta) := \frac{1}{2d}\wh{\phi}(\frac{\eta}{2d\delta^2})\prod_{i = 1}^{d}\wh{\phi}(\frac{\xi_i}{\delta}).
\end{align*}
and for each $\ta_{\om}$, let
\begin{align*}
\psi_{\ta_\om}(\xi, \eta) := \psi_{\ta_{\delta}}(M_{\om}^{-1}((\xi, \eta) - (\om, |\om|^2)))
\end{align*}
where $\phi$ is as defined as in \eqref{schwartzphi}.
Since $\wh{\phi}$ is supported in $[-1, 1]$
and hence $\psi_{\ta_{\delta}}$ is supported in $\ta_{\delta}$.
By \eqref{taqtad}, this then implies that $\psi_{\ta_\om}$ 
is supported in $\ta_{\om}$. Note that $\psi_{\ta_{\om}} = \psi_{\ta_{\delta}}$
if $\om = 0$.

The Fourier inverse of $\psi_{\ta_{\delta}}$ is essentially $|\ta_{\delta}|1_{T_{\delta}}$
and decays rapidly on dilates of $T_{\delta}$. This can be formalized
by the following lemma.
\begin{lemma}\label{dualtad}
For any $E > 0$, we have
\begin{align*}
1_{T_{\delta}}(x, t) \leq \delta^{-(d + 2)}\wc{\psi}_{\ta_{\delta}}(x, t) \lsm_{E} 1_{2T_{\delta}}(x, t) + \sum_{n \geq 1}2^{-Edn}1_{2^{n + 1}T_{\delta} \bs 2^{n}T_{\delta}} (x, t).
\end{align*}
\end{lemma}
\begin{proof}
We have that 
\begin{align}\label{dualtfinv}
\delta^{-(d + 2)}\wc{\psi}_{\ta_{\delta}}(\xi, \eta) = \phi(2d\delta^{2}t)\prod_{i = 1}^{d}\phi(\delta x_i).
\end{align}
By \eqref{schwartzphi}, this is
\begin{align*}
\geq 1_{[-1, 1]}(2d\delta^{2}t)\prod_{i = 1}^{d}1_{[-1, 1]}(\delta x_i) = 1_{T_{\delta}}(x, t).
\end{align*}
On the other hand, combining \eqref{dualtfinv} with the upper bound
in \eqref{schwartzphi}, gives an upper bound of
\begin{align*}
\lsm_{E} (1 + |2d\delta^{2}t|)^{-E}\prod_{i = 1}^{d}(1 + |\delta x_i|)^{-E}
\end{align*}
from which we see that it is $\lsm_{E} 1$ on $2T_{\delta}$
and $\lsm_{E} 2^{-Edn}$ on $2^{n + 1}T_{\delta}\bs 2^{n}T_{\delta}$.
\end{proof}

Let 
\begin{align*}
T_{\om}^{0} := M_{\om}^{-T}T_{\delta} = \{(x, t) \in \R^{d + 1}: |x_i + 2\om_i t| \leq \delta^{-1} \text{ for } i = 1, 2, \ldots, d, |t| \leq (2d)^{-1}\delta^{-2}\}
\end{align*}
be the tube ``dual" to $\ta_{\om}$
centered at the origin.
Heuristically, the Fourier inverse of $\psi_{\ta_{\om}}$ is
$|\ta_{\om}|e(\om\cdot x + |\om|^{2}t)$ multiplied by a function
that is essentially $1_{T_{\om}^{0}}$ and decays rapidly on dilates of $T_{\om}^{0}$.
We can formalize this using the following proposition.
\begin{prop}\label{dualtgen}
We have
$$\delta^{-(d + 2)}\wc{\psi}_{\ta_{\om}}(x, t) = e(\om \cdot x + |\om|^{2}t)\vp_{T_{\om}^{0}}(x, t)$$
where $\vp_{T_{\om}^{0}}$ satisfies
\begin{align*}
1_{T_{\om}^{0}}(x, t) \leq \vp_{T_{\om}^{0}}(x, t) \lsm_{E} 1_{2T_{\om}^{0}}(x, t) + \sum_{n \geq 1}2^{-Edn}1_{2^{n + 1}T_{\om}^{0} \bs 2^{n}T_{\om}^{0}}(x, t)
\end{align*}
for any $E > 0$.
\end{prop}
\begin{proof}
A change of variables gives
\begin{align*}
\delta^{-(d + 2)}\wc{\psi}_{\ta_{\om}}(x, t) = e(\om \cdot x + |\om|^{2}t)\delta^{-(d + 2)}\wc{\psi}_{\ta_{\delta}}(M_{\om}^{T}(x, t)).
\end{align*}
We let $\vp_{T_{\om}^{0}}(x, t) := \delta^{-(d + 2)}\wc{\psi}_{\ta_{\delta}}(M_{\om}^{T}(x, t))$.
Applying Lemma \ref{dualtad} and using the definition of $T_{\om}^{0}$
finishes the proof.
\end{proof}

We are now ready to prove \eqref{lowerboundsum}. Strictly speaking in our
calculations below we assume $q$ and $r$ are finite, but it is not hard to modify
our calculations to work when $q$ or $r$ is infinite. We leave this modification
to the interested reader.

\subsection{Necessity of $\frac{2}{q} + \frac{d}{r} \geq \frac{d}{2}$}\label{sec:line}
In this section we will show that
\begin{align}\label{strichartzlinelowerbound}
\mb{D}_{q, r}(\delta) \gtrsim_{q, r} \delta^{-(\frac{d}{2} - \frac{d}{r} - \frac{2}{q})}
\end{align}
for all sufficiently small $\delta$ depending on $d$.
This example is the standard bush example from Bourgain and Demeter's $(l^{2}, L^{p}_{x, t})$ decoupling for the paraboloid showing that to have $D_{p, p}(\delta) \lsm_{p, \vep} \delta^{-\vep}$, one must have $2 \leq p \leq \frac{2(d + 2)}{d}$.

We first claim that $B(0, 10^{-10}) \subset T_{\om}^{0}$ for all $\om \in \Xi$ if
$\delta$ is sufficiently small depending on $d$.
Since $T_{\om}^{0} = M_{\om}^{-T}T_{\delta}$, we need to show that $M_{\om}^{T}B(0, 10^{-10}) \subset T_{\delta}$.
But this follows from that if $(x, t) \in B(0, 10^{-10})$, then $|x_i + 2\om_i t| \leq 3 \cdot 10^{-10} \leq \delta^{-1}$ and $|t| \leq 10^{-10} \leq (2d)^{-1}\delta^{-2}$
for all $\delta$ sufficiently small depending on $d$.

Let $f_{\ta_\om}(x, t)=\delta^{-(d+2)}\wc{\psi}_{\ta_\om}(x, t) = e(\om \cdot x + |\om|^2 t)\vp_{T_{\om}^{0}}(x, t).$
Since $|\om \cdot x + |\om|^2 t| \leq 2 \cdot 10^{-10}$ for $(x, t) \in B(0, 10^{-10})$, we have $\cos(2\pi(\om \cdot x + |\om|^2 t)) \gtrsim 1$ for such $(x, t)$ and hence
\begin{align*}
|\sum_{\om \in \Xi}f_{\ta_\om}(x, t)| \geq |\sum_{\om \in \Xi}\cos(2\pi (\om \cdot x + |\om|^2 t))\vp_{T_{\om}^{0}}(x, t)| \geq |\sum_{\om \in \Xi}\cos(2\pi (\om \cdot x + |\om|^2 t))1_{T_{\om}^{0}}(x, t)|\gtrsim \# \Xi
\end{align*}
for all $(x, t) \in B(0, 10^{-10})$.
This implies that $\nms{\sum_{\om \in \Xi}f_{\ta_\om}}_{L^{q}_{t}L^{r}_{x}} \gtrsim_{q, r} \# \Xi$. On the other hand,  by choosing $E$ sufficiently large, Proposition \ref{dualtgen} implies that
$\nms{f_{\ta_{\om}}}_{L^{q}_{t}L^{r}_{x}} \lsm_{q, r} \delta^{-d/r - 2/q}$ and hence $(\sum_{\om \in \Xi}\nms{f_{\ta_\om}}_{L^{q}_{t}L^{r}_{x}}^{2})^{1/2} \lesssim_{q, r} \delta^{-d/r - 2/q}(\# \Xi)^{1/2}$.
This implies that
$(\# \Xi)^{1/2} \delta^{d/r + 2/q}\lsm_{q, r} \mb{D}_{q, r}(\delta)$.
Using the fact that $\# \Xi \sim \delta^{-d}$ we obtain \eqref{strichartzlinelowerbound}.

\subsection{Necessity of $r \geq 2$}\label{sec:r2}
In this section we will show that
\begin{align}\label{r2lower}
\mb{D}_{q, r}(\delta) \gtrsim_{q, r} \delta^{-d(\frac{1}{r} - \frac{1}{2})}
\end{align}
for all $\delta$ sufficiently small depending on $d$.

For each $\om \in \Xi$, choose points $(x_\om, 0)$ which are $\gtrsim \delta^{-4}$ separated. Let $f_{\ta_\om}$ be the Fourier inverse of
$\delta^{-(d + 2)}e((-x_\om, 0) \cdot (\xi, \eta))\psi_{\ta_\om}(\xi, \eta)$. 
Then we have
$$f_{\ta_\om}(x, t) = e(\om \cdot (x - x_\om) + |\om|^{2}t)\vp_{T_{\om}^{0} + (x_\om, 0)}(x, t)$$
where by Proposition \ref{dualtgen},
\begin{align}
\begin{aligned}\label{star}
1_{T_{\om}^{0} + (x_\om, 0)}(x, t) &\leq \vp_{T_{\om}^{0} + (x_\om, 0)}(x, t)\\
&\lsm_{E} 1_{2T_{\om}^{0} + (x_\om, 0)}(x, t) + \sum_{n \geq 1}2^{-Edn}1_{(2^{n + 1}T_{\om}^{0} + (x_\om, 0))\bs (2^{n}T_{\om}^{0} + (x_\om, 0))}(x, t)
\end{aligned}
\end{align}
for $E > 0$ which we are free to choose.
Choosing $E$ sufficiently large, this formula immediately implies that $\nms{f_{\ta_\om}}_{L^{q}_{t}L^{r}_{x}} \lsm_{q, r} \delta^{-d/r - 2/q}$ and hence
\begin{align}\label{r2eq3}
(\sum_{\om \in \Xi}\nms{f_{\ta_\om}}_{L^{q}_{t}L^{r}_{x}}^{2})^{1/2} \lsm_{q, r} \delta^{-d/r - 2/q}(\# \Xi)^{1/2}.
\end{align}
Next, the triangle inequality gives that
\begin{align*}
\nms{\sum_{\om \in \Xi}f_{\ta_\om}}_{L^{q}_{t}L^{r}_{x}} &\geq \nms{\sum_{\om \in \Xi}f_{\ta_\om}1_{\delta^{-1}T_{\om}^{0} + (x_\om, 0)}}_{L^{q}_{t}L^{r}_{x}} - \nms{\sum_{n: 2^{n} \gtrsim \delta^{-1}}\sum_{\om \in \Xi}f_{\ta_\om}1_{(2^{n + 1}T_{\om}^{0} + (x_\om, 0))\bs (2^{n}T_{\om}^{0} + (x_\om, 0))}}_{L^{q}_{t}L^{r}_{x}}\\
& := I - II.
\end{align*}
We first analyze the main term I. By how $f_{\ta_\om}$ is defined, it is equal to
\begin{align*}
\nms{\sum_{\om \in \Xi}e(\om \cdot (x - x_\om) + |\om|^{2}t)\vp_{T_{\om}^{0} + (x_\om, 0)}(x, t)1_{\delta^{-1}T_{\om}^{0} + (x_\om, 0)}(x, t)}_{L^{q}_{t}L^{r}_{x}}.
\end{align*}
Since the $x_\om$ are $\gtrsim \delta^{-4}$ separated and $T_{\om}^{0}$ is 
essentially a rectangle box of dimension
$O(\delta^{-1}) \times \cdots \times O(\delta^{-1}) \times O(\delta^{-2})$ centered at the origin, 
the $\{\delta^{-1}T_{\om}^{0} + (x_\om, 0)\}_{\om}$ are disjoint.
Therefore
\begin{align*}
|\sum_{\om \in \Xi}e(\om \cdot (x - x_\om) + |\om|^{2}t)&\vp_{T_{\om}^{0} + (x_\om, 0)}(x, t)1_{\delta^{-1}T_{\om}^{0} + (x_\om, 0)}(x, t)|^{r}\\
&= \sum_{\om \in \Xi}\vp_{T_{\om}^{0} + (x_\om, 0)}(x, t)^{r}1_{\delta^{-1}T_{\om}^{0} + (x_\om, 0)}(x, t) \geq \sum_{\om \in \Xi}1_{T_{\om}^{0} + (x_\om, 0)}(x, t)
\end{align*}
and hence  
\begin{align}\label{r2eq2}
I &\geq (\int_{\R}(\int_{\R^d}\sum_{\om \in \Xi}1_{T_{\om}^{0} + (x_\om, 0)}(x, t)\, dx)^{q/r}\, dt)^{1/q}\nonumber\\
&= (\int_{\R}(\sum_{\om \in \Xi}\int_{\R^d}\prod_{i = 1}^{d}1_{|x_i + 2\om_i t| \leq \delta^{-1}}\, dx)^{q/r}\, 1_{|t| \lsm \delta^{-2}}\, dt)^{1/q} \gtrsim_{q, r} \delta^{-d/r - 2/q}(\# \Xi)^{1/r}.
\end{align}

Now we analyze the error term II.
We use trivial bounds and \eqref{star} to obtain that it is
\begin{align}\label{r2error}
&\leq (\# \Xi)\sum_{n: 2^{n} \gtrsim \delta^{-1}}\max_{\om \in \Xi}\nms{\vp_{T_{\om}^{0} + (x_\om, 0)}1_{(2^{n + 1}T_{\om}^{0} + (x_\om, 0))\bs (2^{n}T_{\om}^{0} + (x_\om, 0))}}_{L^{q}_{t}L^{r}_{x}}\nonumber\\
&\lsm_{E} (\# \Xi)\sum_{n: 2^{n} \gtrsim \delta^{-1}}2^{-Edn}\max_{\om \in \Xi}\nms{1_{2^{n + 1}T_{\om}^{0} + (x_\om, 0))}}_{L^{q}_{t}L^{r}_{x}}\nonumber\\
&\lsm_{q, r} \delta^{-d/r - 2/q}(\# \Xi)\sum_{n: 2^{n} \gtrsim \delta^{-1}}2^{-Edn}(2^{n + 1})^{\frac{d}{r} + \frac{1}{q}}.
\end{align}
Choosing $E$ sufficiently large and using that $\# \Xi \sim \delta^{-d}$, we have
\begin{align}\label{r2geom}
(\# \Xi)\sum_{n: 2^{n} \gtrsim \delta^{-1}}2^{-Edn}(2^{n + 1})^{\frac{d}{r} + \frac{1}{q}} \lsm \delta^{-d}\sum_{n: 2^{n} \gtrsim \delta^{-1}}2^{-Edn}(2^{n + 1})^{d + 1} \lsm \delta^{100d}.
\end{align}
Combining this estimate with \eqref{r2eq2} gives that 
\begin{align*}
I - II \gtrsim_{q, r} \delta^{-d/r - 2/q}((\# \Xi)^{1/r} - O(\delta^{100d})).
\end{align*} 
Comparing this with \eqref{r2eq3} then shows that
\begin{align*}
\mb{D}_{q, r}(\delta) \gtrsim_{q, r} (\# \Xi)^{\frac{1}{r} - \frac{1}{2}} - O(\delta^{100d})(\# \Xi)^{-1/2} \gtrsim \delta^{-d(\frac{1}{r} - \frac{1}{2})}
\end{align*}
for all sufficiently small $\delta$ which completes the proof of \eqref{r2lower}.

\subsection{Necessity of $q \geq 2$}\label{sec:q2}
Here we will show that
\begin{align}\label{q2lower}
\mb{D}_{q, r}(\delta) \gtrsim_{q, r} \delta^{-d(\frac{1}{q} - \frac{1}{2})}
\end{align}
for all $\delta$ sufficiently small depending on $d$.
This argument is similar to that for \eqref{r2lower} except
now we choose points spaced out in time.

For each $\om \in \Xi$, choose points $(0, t_\om)$ which are $\gtrsim \delta^{-4}$ separated. Let $f_{\ta_\om}$ be the Fourier inverse of
$\delta^{-(d + 2)}e((0, -t_\om) \cdot (\xi, \eta))\psi_{\ta_\om}(\xi, \eta)$, then
$$
f_{\ta_\om}(x, t) =  e(\om \cdot x + |\om|^{2}(t - t_\om))\vp_{T_{\om}^{0} + (0, t_\om)}(x, t),
$$
where 
\begin{align}
\begin{aligned}\label{2star}
1_{T_{\om}^{0} + (0, t_\om)}(x, t) &\leq \vp_{T_{\om}^{0} + (0, t_\om)}(x, t)\\
&\lsm_{E} 1_{2T_{\om}^{0} + (0, t_\om)}(x, t) + \sum_{n \geq 1}2^{-Edn}1_{(2^{n + 1}T_{\om}^{0} + (0, t_\om))\bs (2^{n}T_{\om}^{0} + (0, t_\om))}(x, t)
\end{aligned}
\end{align}
for $E > 0$ which we are free to choose.
As before, choosing $E$ sufficiently large gives $\nms{f_{\ta_\om}}_{L^{q}_{t}L^{r}_{x}} \lsm_{q, r} \delta^{-d/r - 2/q}$ and hence
\begin{align}\label{q2eq3}
(\sum_{\om \in \Xi}\nms{f_{\ta_\om}}_{L^{q}_{t}L^{r}_{x}}^{2})^{1/2} \lsm_{q, r} \delta^{-d/r - 2/q}(\# \Xi)^{1/2}.
\end{align}
Next, the triangle inequality gives
\begin{align*}
\nms{\sum_{\om \in \Xi}f_{\ta_\om}}_{L^{q}_{t}L^{r}_{x}} &\geq \nms{\sum_{\om \in \Xi}f_{\ta_\om}1_{\delta^{-1}T_{\om}^{0} + (0, t_\om)}}_{L^{q}_{t}L^{r}_{x}} - \nms{\sum_{n: 2^{n} \gtrsim \delta^{-1}}\sum_{\om \in \Xi}f_{\ta_\om}1_{(2^{n + 1}T_{\om}^{0} + (0, t_\om))\bs (2^{n}T_{\om}^{0} + (0, t_\om))}}_{L^{q}_{t}L^{r}_{x}}\\
& := I - II.
\end{align*}
We first analyze the main term $I$. Substituting the definition of $f_{\ta_\om}$, we have that it is equal to
\begin{align*}
\nms{\sum_{\om \in \Xi}e(\om \cdot x + |\om|^{2}(t - t_\om))\vp_{T_{\om}^{0} + (0, t_\om)}(x, t)1_{\delta^{-1}T_{\om}^{0} + (0, t_\om)}(x, t)}_{L^{q}_{t}L^{r}_{x}}.
\end{align*}
Since the $t_\om$ are $\gtrsim \delta^{-4}$ separated, as in the proof of \eqref{r2lower}, the $\delta^{-1}T_{\om}^{0} + (0, t_\om)$ are pairwise disjoint.
Therefore
\begin{align*}
|\sum_{\om \in \Xi}e(\om \cdot x + |\om|^{2}(t - t_{\om}))&\vp_{T_{\om}^{0} + (0, t_\om)}(x, t)1_{\delta^{-1}T_{\om}^{0} + (0, t_\om)}(x, t)|^{r}\\
&= \sum_{\om \in \Xi}\vp_{T_{\om}^{0} + (0, t_\om)}(x, t)^{r}1_{\delta^{-1}T_{\om}^{0} + (0, t_\om)}(x, t) \geq \sum_{\om \in \Xi}1_{T_{\om}^{0} + (0, t_\om)}(x, t)
\end{align*}
and hence
\begin{align*}
I &\geq (\int_{\R}(\int_{\R^d}\sum_{\om \in \Xi}1_{T_{\om}^{0} + (0, t_\om)}(x, t)\, dx)^{q/r}\, dt)^{1/q}= (\int_{\R}(\sum_{\om \in \Xi}1_{|t - t_{\om}| \lsm \delta^{-2}}\int_{\R^d}\prod_{i = 1}^{d}1_{|x_i + 2\om_i t| \leq \delta^{-1}}\, dx)^{q/r})^{1/q}.
\end{align*}
Fixing $t$ and applying a change of variables $x$ gives that the above
is
\begin{align}\label{q2eq2}
= \delta^{-d/r}(\int_{\R}(\sum_{\om \in \Xi}1_{|t - t_{\om}| \lsm \delta^{-2}})^{q/r}\, dt)^{1/q} = \delta^{-d/r}(\int_{\R}\sum_{\om \in \Xi}1_{|t - t_{\om}| \lsm \delta^{-2}}\, dt)^{1/q} \sim_{q, r} \delta^{-d/r - 2/q}(\#\Xi)^{1/q}
\end{align}
where here we have used that the $t_{\om}$ are $\gtrsim \delta^{-4}$ separated.

Now we analyze the error term II.
We use trivial bounds and \eqref{2star} to obtain that it is
\begin{align*}
&\leq (\# \Xi)\sum_{n: 2^{n} \gtrsim \delta^{-1}}\max_{\om \in \Xi}\nms{\vp_{T_{\om}^{0} + (0, t_\om)}1_{(2^{n + 1}T_{\om}^{0} + (0, t_\om))\bs (2^{n}T_{\om}^{0} + (0, t_\om))}}_{L^{q}_{t}L^{r}_{x}}\\
&\lsm_{E} (\# \Xi)\sum_{n: 2^{n} \gtrsim \delta^{-1}}2^{-Edn}\max_{\om \in \Xi}\nms{1_{2^{n + 1}T_{\om}^{0} + (0, t_\om)}}_{L^{q}_{t}L^{r}_{x}}.
\end{align*}
Proceeding in the same way as estimating \eqref{r2error} and \eqref{r2geom} shows that
if we choose $E$ sufficiently large, we have $II \lsm_{q, r}\delta^{-d/r - 2/q}\delta^{100d}$.
Combining this with \eqref{q2eq2} gives that 
\begin{align*}
I - II \gtrsim_{q, r} \delta^{-d/r - 2/q}((\# \Xi)^{1/q} - O(\delta^{100d})).
\end{align*} 
Comparing this with \eqref{q2eq3} then shows that
\begin{align*}
\mb{D}_{q, r}(\delta) \gtrsim_{q, r} (\# \Xi)^{\frac{1}{q} - \frac{1}{2}} - O(\delta^{100d})(\# \Xi)^{-1/2} \gtrsim \delta^{-d(\frac{1}{q} - \frac{1}{2})}
\end{align*}
for all sufficiently small $\delta$ which completes the proof of \eqref{q2lower}.
\subsection{Necessity of $r \leq \frac{2(d + 2)}{d}$}\label{sec:r2d2d}
Finally, we will show that we have
\begin{align}\label{lowerbddlarge}
\mb{D}_{q, r}(\delta) \gtrsim_{q, r} \delta^{-(\frac{d}{2} - \frac{d + 2}{r})}
\end{align}
for all $\delta$ sufficiently small depending on $d$.

Let $\vep_0 := 1/2$. The only property that we will need about $\vep_0$ is that $0 < \vep_{0}^2 < \vep_0$. For integers $0 \leq k_1 < \delta^{-2}$ and $0 \leq k_i < \delta^{-1}$, $i = 2, \ldots, d$, let $c_{\vec{k}} := (k_1\delta^{-1 - \vep_0}, k_2\delta^{-1 - \vep_0}, \ldots, k_{d}\delta^{-1 - \vep_0}, k_1) = (\vec{k}\delta^{-1 - \vep_0}, k_1)$
and $C_{\vec{k}} := B(c_{\vec{k}}, 10^{-10})$.
Our example will consist of $\delta^{-(d + 1)}$ many bush examples from Section \ref{sec:line} each tuned so that they have constructive interference on 
the $\{C_{\vec{k}}\}$.

\subsubsection{Properties of $\{C_{\vec{k}}\}$}
By construction, the $\{C_{\vec{k}}\}$ are disjoint. We will need the following additional basic geometric properties of $\{C_{\vec{k}}\}$.

\begin{lemma}\label{Cveckcontain}
For each $\vec{k}$, we have $C_{\vec{k}} \subset T_{\om}^{0} + c_{\vec{k}}$.
\end{lemma}
\begin{proof}
This follows from that $B(0, 10^{-10}) \subset T_{\om}^{0}$ as shown in Section \ref{sec:line}.
\end{proof}

\begin{lemma}\label{2ndilatehd}
Suppose $\vec{k} \neq \vec{k}'$. For all sufficiently small
$\delta$ (depending on $d$) and $M$ such that $M \lsm \delta^{-\vep_{0}^2}$,
we have $(MT_{\om}^{0} + c_{\vec{k}}) \cap (MT_{\om}^{0} + c_{\vec{k}'}) = \emptyset$.
\end{lemma}
\begin{proof}
Suppose there was an $(x, t) \in (MT_{\om}^{0} + c_{\vec{k}}) \cap (MT_{\om}^{0} + c_{\vec{k}'})$. Then
$|(x_i - k_i \delta^{-1 - \vep_0}) + 2\om_i (t - k_1)| \leq M\delta^{-1}$
and
$|(x_i - k_{i}'\delta^{-1 - \vep_0}) + 2\om_i (t - k_{1}'
)| \leq M\delta^{-1}$
for $i = 1, 2, \ldots, d$.
Combining these two sets of inequalities gives
\begin{align}\label{kkp}
|(k_{i} - k_{i}')\delta^{-1 - \vep_0} + 2\om_i (k_1 - k_{1}')| \leq 2M\delta^{-1}
\end{align}
for $i = 1, 2, \ldots, d$.
If $k_{1} \neq k_{1}'$, then $|k_1 - k_{1}'| \geq 1$ and looking at the $i = 1$ case of \eqref{kkp} gives
\begin{align*}
|\delta^{-1 - \vep_0} + 2\om_1| \leq |k_{1} - k_{1}'||\delta^{-1 - \vep_0} + 2\om_1| \lsm \delta^{-1 - \vep_0^2}.
\end{align*}
But by our choice of $\vep_0$, this cannot happen for all $\delta$ sufficiently small.
Thus we must have $k_1 = k_1'$. This reduces \eqref{kkp} to
$|k_{i} - k_{i}'| \leq 2M\delta^{\vep_0} \lsm \delta^{\vep_0 - \vep_{0}^2}$ for $i = 2, \ldots, d$. But if $\delta$ is sufficiently small, then we must have
$|k_i - k_{i}'| = 0$ as both $k_i$ and $k_{i}'$ are integers.
But this contradicts that $\vec{k} \neq \vec{k}'$.
Therefore if $\vec{k} \neq \vec{k}'$, $(MT_{\om}^{0} + c_{\vec{k}}) \cap (MT_{\om}^{0} + c_{\vec{k}'})$ must be empty. This completes the proof of the lemma.
\end{proof}

\begin{lemma}\label{keycjkl}
For all sufficiently small $\delta$ (depending on $d$) and $M$ such that $M \lsm \delta^{-\vep_0^2}$, we have $C_{\vec{k}'} \subset MT_{\om}^{0} + c_{\vec{k}}$ if and only if $\vec{k}' = \vec{k}$.
\end{lemma}
\begin{proof}
If $\vec{k} = \vec{k}'$, then the conclusion follows from Lemma \ref{Cveckcontain}. 
On the other hand, if $C_{\vec{k}'} \subset MT_{\om}^{0} + c_{\vec{k}}$, then by Lemma \ref{Cveckcontain}, $(MT_{\om}^{0} + c_{\vec{k}'}) \cap (MT_{\om}^{0} + c_{\vec{k}}) \neq \emptyset$ which by Lemma \ref{2ndilatehd} implies $\vec{k} = \vec{k}'$.
This completes the proof of the lemma.
\end{proof}

\subsubsection{The example}
Let $f_{\ta_\om}$ be the Fourier inverse of
\begin{align*}
\delta^{-(d + 2)}\sum_{\st{0 \leq k_1 < \delta^{-2}\\0 \leq k_i < \delta^{-1}\\i = 2, \ldots, d}}  e(-c_{\vec{k}} \cdot (\xi, \eta)) \psi_{\ta_\om}(\xi, \eta)
\end{align*}
where $\xi = (\xi_1, \ldots, \xi_{d})$ and $\eta \in \R$.
Therefore
\begin{align*}
f_{\ta_\om}(x, t) = \sum_{\st{0 \leq k_1 < \delta^{-2}\\0 \leq k_i < \delta^{-1}\\i = 2, \ldots, d}}e(\om \cdot (x - \vec{k}\delta^{-1 - \vep_0}) + |\om|^{2}(t - k_1))\vp_{T_{\om}^{0} + c_{\vec{k}}}(x, t).
\end{align*}
We then write
\begin{align*}
\sum_{\om \in \Xi}f_{\ta_\om}(x, t) &= \sum_{\om \in \Xi}\sum_{\st{0 \leq k_1 < \delta^{-2}\\0 \leq k_i < \delta^{-1}\\i = 2, \ldots, d}}e(\om \cdot (x - \vec{k}\delta^{-1 - \vep_0}) + |\om|^{2}(t - k_1))\vp_{T_{\om}^{0} + c_{\vec{k}}}(x, t)1_{\delta^{-\vep_0^2}T_{\om}^{0} + c_{\vec{k}}}(x, t)\\
&\hspace{0.25in} + \sum_{\om \in \Xi}\sum_{\st{0 \leq k_1 < \delta^{-2}\\0 \leq k_i < \delta^{-1}\\i = 2, \ldots, d}}e(\om \cdot (x - \vec{k}\delta^{-1 - \vep_0}) + |\om|^{2}(t - k_1))\vp_{T_{\om}^{0} + c_{\vec{k}}}(x, t)1_{\R^{d+1} \bs (\delta^{-\vep_0^2}T_{\om}^{0} + c_{\vec{k}})}(x, t)\\
&:= I + II
\end{align*}
and hence
\begin{align*}
\nms{\sum_{\om \in \Xi}f_{\ta_\om}}_{L^{q}_{t}L^{r}_{x}} \geq \nms{I}_{L^{q}_{t}L^{r}_{x}} - \nms{II}_{L^{q}_{t}L^{r}_{x}}.
\end{align*}
We first have
\begin{align}\label{IIhd1}
\nms{II}_{L^{q}_{t}L^{r}_{x}}&\lsm (\# \Xi)\delta^{-(d+1)}\max_{\om, \vec{k}}\nms{\vp_{T_{\om}^{0} + c_{\vec{k}}}(x, t)1_{\R^{d+1} \bs (\delta^{-\vep_0^2}T_{\om}^{0} + c_{\vec{k}})}}_{L^{q}_{t}L^{r}_{x}}\nonumber\\
&\lsm_{E} (\# \Xi)\delta^{-(d+1)}\max_{\om, \vec{k}}\sum_{n: 2^{n} \gtrsim \delta^{-\vep_0^2}}2^{-Edn}\nms{1_{(2^{n + 1}T_{\om}^{0} + c_{\vec{k}}) \bs (2^{n}T_{\om}^{0} + c_{\vec{k}})}}_{L^{q}_{t}L^{r}_{x}}\nonumber\\
&\lsm_{q, r} (\# \Xi) \delta^{-(d+1) - d/r - 2/q}\sum_{n: 2^{n} \gtrsim \delta^{-\vep_0^2}}2^{-Edn}(2^{n + 1})^{\frac{d}{r} + \frac{1}{q}}\nonumber\\
&\lsm_{q, r} \delta^{-3d- 3}\sum_{n: 2^{n} \gtrsim \delta^{-\vep_0^2}}2^{-(Ed - d - 1)n}
\end{align}
where in the second inequality we have used Proposition \ref{dualtgen} and in the last inequality we have used that $q, r \geq 1$ and $\# \Xi \sim \delta^{-d}$.
Choosing $E = 2000/\vep_{0}^2 + (d + 1)/d$ then shows that the above is
$\lsm_{q, r} \delta^{1000d}$.

Next, we have
\begin{align*}
&\|I\|_{L^{q}_{t}L^{r}_{x}} \geq \nms{I \cdot 1_{\bigcup_{\vec{k}}}C_{\vec{k}}}_{L^{q}_{t}L^{r}_{x}}\\
& = \nms{\sum_{\om \in \Xi}\sum_{\st{0 \leq k_1 < \delta^{-2}\\0 \leq k_i < \delta^{-1}\\i = 2, \ldots, d}}\sum_{\st{0 \leq k_{1}' < \delta^{-2}\\0 \leq k_{i}' < \delta^{-1}\\i = 2, \ldots, d}}e(\om \cdot (x - \vec{k}\delta^{-1 - \vep_0}) + |\om|^{2}(t - k_1))\vp_{T_{\om}^{0} + c_{\vec{k}}}(x, t)1_{\delta^{-\vep_0^2}T_{\om}^{0} + c_{\vec{k}}}(x, t)1_{C_{\vec{k}'}}}_{L^{q}_{t}L^{r}_{x}}.
\end{align*}
By Lemma \ref{keycjkl}, for fixed $\vec{k}$, the only $\{C_{\vec{k}'}\}$ that intersects $\delta^{-\vep_{0}^2}T_{\om}^{0} + c_{\vec{k}}$ is $C_{\vec{k}}$. Therefore the above is equal to
\begin{align*}
&\nms{\sum_{\om \in \Xi}\sum_{\st{0 \leq k_1 < \delta^{-2}\\0 \leq k_i < \delta^{-1}\\i = 2, \ldots, d}}e(\om \cdot (x - \vec{k}\delta^{-1 - \vep_0}) + |\om|^{2}(t - k_1))\vp_{T_{\om}^{0} + c_{\vec{k}}}(x, t)1_{\delta^{-\vep_0^2}T_{\om}^{0} + c_{\vec{k}}}(x, t)1_{C_{\vec{k}}}(x, t)}_{L^{q}_{t}L^{r}_{x}}\\
&\geq \nms{\sum_{\om \in \Xi}\sum_{\st{0 \leq k_1 < \delta^{-2}\\0 \leq k_i < \delta^{-1}\\i = 2, \ldots, d}}\cos(2\pi(\om \cdot (x - \vec{k}\delta^{-1 - \vep_0}) + |\om|^{2}(t - k_1)))\vp_{T_{\om}^{0} + c_{\vec{k}}}(x, t)1_{\delta^{-\vep_0^2}T_{\om}^{0} + c_{\vec{k}}}(x, t)1_{C_{\vec{k}}}(x, t)}_{L^{q}_{t}L^{r}_{x}}.
\end{align*}
Since $|\om \cdot (x - \vec{k}\delta^{-1 - \vep_0}) + |\om|^2 (t - k_1)| \leq 2 \cdot 10^{-10}$ for $(x, t) \in C_k$, we have $\cos(2\pi(\om \cdot (x - \vec{k}\delta^{-1 - \vep_0}) + |\om|^{2}(t - k_1))) \gtrsim 1$ for such $(x, t)$.
Therefore the above is
\begin{align}\label{3star}
\gtrsim_{q, r} \nms{\sum_{\om \in \Xi}\sum_{\st{0 \leq k_1 < \delta^{-2}\\0 \leq k_i < \delta^{-1}\\i = 2, \ldots, d}}1_{T_{\om}^{0} + c_{\vec{k}}}(x, t)1_{\delta^{-\vep_0^2}T_{\om}^{0} + c_{\vec{k}}}(x, t)1_{C_{\vec{k}}}(x, t)}_{L^{q}_{t}L^{r}_{x}} \geq (\# \Xi)\nms{\sum_{\st{0 \leq k_1 < \delta^{-2}\\0 \leq k_i < \delta^{-1}\\i = 2, \ldots, d}}1_{C_{\vec{k}}}(x, t)}_{L^{q}_{t}L^{r}_{x}}
\end{align}
where in the second inequality we have made use of Lemma \ref{Cveckcontain}.
Since the $\{C_{\vec{k}}\}$ are disjoint,
\begin{align*}
|\sum_{\st{0 \leq k_1 < \delta^{-2}\\0 \leq k_i < \delta^{-1}\\i = 2, \ldots, d}}1_{C_{\vec{k}}}(x, t)|^{r} = \sum_{\st{0 \leq k_1 < \delta^{-2}\\0 \leq k_i < \delta^{-1}\\i = 2, \ldots, d}}1_{C_{\vec{k}}}(x, t) \geq \sum_{\st{0 \leq k_1 < \delta^{-2}\\0 \leq k_i < \delta^{-1}\\i = 2, \ldots, d}}1_{|t - k_1| \leq 10^{-11}}\prod_{i = 1}^{d}1_{|x_i - k_{i}\delta^{-1 - \vep_0}| \leq 10^{-11}/\sqrt{d}}
\end{align*}
which implies that
\begin{align*}
\nms{\sum_{\st{0 \leq k_1 < \delta^{-2}\\0 \leq k_i < \delta^{-1}\\i = 2, \ldots, d}}1_{C_{\vec{k}}}}_{L^{q}_{t}L^{r}_{x}} &\geq (\int_{\R}(\int_{\R^d}\sum_{\st{0 \leq k_1 < \delta^{-2}\\0 \leq k_i < \delta^{-1}\\i = 2, \ldots, d}}1_{|t - k_1| \leq 10^{-11}}\prod_{i = 1}^{d}1_{|x_i - k_{i}\delta^{-1 - \vep_0}| \leq 10^{-11}/\sqrt{d}}\, dx)^{q/r}\, dt)^{1/q}\\
&=(\int_{\R}(\sum_{0 \leq k_1 < \delta^{-2}}1_{|t - k_1| \leq 10^{-11}}\sum_{\st{0 \leq k_i < \delta^{-1}\\i = 2, 3, \ldots, d}}\int_{\R^d}\prod_{i = 1}^{d}1_{|x_i - k_{i}\delta^{-1 - \vep_0}| \leq 10^{-11}/\sqrt{d}}\, dx)^{q/r}\, dt)^{1/q}\\
&\gtrsim_{q, r} \delta^{-\frac{d - 1}{r}}(\int_{\R}(\sum_{0 \leq k_1 < \delta^{-2}}1_{|t - k_1| \leq 10^{-11}})^{q/r}\, dt)^{1/q}\\
&= \delta^{-\frac{d - 1}{r}}(\int_{\R}1_{\bigcup_{0 \leq k_1 < \delta^{-2}}1_{|t - k_1| \leq 10^{-11}}}\, dt)^{1/q} \gtrsim_{q, r}\delta^{-\frac{d - 1}{r} - \frac{2}{q}}.
\end{align*}
Thus inserting this into \eqref{3star}, we have obtained that for all $\delta$ sufficiently small, we have
\begin{align}\label{tuninglowerbdhd}
\nms{\sum_{\om \in \Xi}f_{\ta_{\om}}}_{L^{q}_{t}L^{r}_{x}} \gtrsim_{q, r} (\# \Xi)\delta^{-\frac{d - 1}{r} - \frac{2}{q}} - \delta^{1000d}.
\end{align}

Next we compute $(\sum_{\om \in \Xi}\nms{f_{\ta_{\om}}}_{L^{q}_{t}L^{r}_{x}}^{2})^{1/2}$. For any $\om$, we have
\begin{align*}
|f_{\ta_{\om}}(x, t)| &= |\sum_{\st{0 \leq k_1 < \delta^{-2}\\0 \leq k_i < \delta^{-1}\\i = 2, \ldots, d}}e(\om \cdot (x - \vec{k}\delta^{-1 - \vep_0}) + |\om|^2 (t - k_1))\vp_{T_{\om}^{0} + c_{\vec{k}}}(x, t)|\\
&\leq \sum_{\st{0 \leq k_1 < \delta^{-2}\\0 \leq k_i < \delta^{-1}\\i = 2, \ldots, d}}\vp_{T_{\om}^{0} + c_{\vec{k}}}(x, t)\lsm_{E}\sum_{\st{0 \leq k_1 < \delta^{-2}\\0 \leq k_i < \delta^{-1}\\i = 2, \ldots, d}}\sum_{n \geq 0}2^{-Edn}1_{2^{n + 1}T_{\om}^{0} + c_{\vec{k}}}(x, t)\\
&= \sum_{n \geq 0}2^{-Edn}\sum_{\st{0 \leq k_1 < \delta^{-2}\\0 \leq k_i < \delta^{-1}\\i = 2, \ldots, d}}1_{2^{n + 1}T_{\om}^{0} + c_{\vec{k}}}(x, t)
\end{align*}
where in the last inequality we have used Proposition \ref{dualtgen}.
Therefore
\begin{align*}
\nms{f_{\ta_{\om}}}_{L^{q}_{t}L^{r}_{x}} &\leq \sum_{n: 2^{n} \lsm \delta^{-\vep_0^2}}2^{-Edn}\nms{\sum_{\st{0 \leq k_1 < \delta^{-2}\\0 \leq k_i < \delta^{-1}\\i = 2, \ldots, d}}1_{2^{n + 1}T_{\om}^{0} + c_{\vec{k}}}}_{L^{q}_{t}L^{r}_{x}} + \sum_{n: 2^{n} \gtrsim \delta^{-\vep_0^2}}2^{-Edn}\nms{\sum_{\st{0 \leq k_1 < \delta^{-2}\\0 \leq k_i < \delta^{-1}\\i = 2, \ldots, d}}1_{2^{n + 1}T_{\om}^{0} + c_{\vec{k}}}}_{L^{q}_{t}L^{r}_{x}}\\
&:= III + IV.
\end{align*}

Treating IV as in how we treated II in \eqref{IIhd1}, we have
\begin{align}\label{ivbd1}
IV \leq \delta^{-(d + 1) - d/r - 2/q}\sum_{n: 2^{n} \gtrsim \delta^{-\vep_0^2}}2^{-Edn}(2^{n + 1})^{\frac{d}{r} + \frac{1}{q}} \lsm_{q, r} \delta^{1000d}
\end{align}
if we choose $E$ sufficiently large. We now turn to III. Since $2^{n} \lsm \delta^{-\vep_0^2}$,
Lemma \ref{2ndilatehd} shows that for fixed $n$, $\{2^{n + 1}T_{\om}^{0} + c_{\vec{k}}\}_{\vec{k}}$ is a disjoint collection of tubes. Therefore
\begin{align*}
|\sum_{\st{0 \leq k_1 < \delta^{-2}\\0 \leq k_i < \delta^{-1}\\i = 2, \ldots, d}}1_{2^{n + 1}T_{\om}^{0} + c_{\vec{k}}}(x, t)|^{r} = \sum_{\st{0 \leq k_1 < \delta^{-2}\\0 \leq k_i < \delta^{-1}\\i = 2, \ldots, d}}1_{2^{n + 1}T_{\om}^{0} + c_{\vec{k}}}(x, t).
\end{align*}
Thus
\begin{align*}
\nms{\sum_{\st{0 \leq k_1 < \delta^{-2}\\0 \leq k_i < \delta^{-1}\\i = 2, \ldots, d}}1_{2^{n + 1}T_{\om}^{0} + c_{\vec{k}}}}_{L^{q}_{t}L^{r}_{x}} &= (\int_{\R}(\sum_{\st{0 \leq k_1 < \delta^{-2}\\0 \leq k_i < \delta^{-1}\\i = 2, \ldots, d}}\int_{\R^d}1_{2^{n + 1}T_{\om}^{0} + c_{\vec{k}}}(x, t)\, dx)^{q/r}\, dt)^{1/q}\\
&=(\int_{\R}(\delta^{-(d - 1)}\sum_{0 \leq k_1 < \delta^{-2}}\int_{\R^d}1_{2^{n + 1}T_{\om}^{0}}(x, t- k_1)\, dx)^{q/r}\, dt)^{1/q}.
\end{align*}
For $(x, t)$ such that $(x, t - k_1) \in 2^{n + 1}T_{\om}^{0}$ and $0 \leq k_1 < \delta^{-2}$, we have that
$(x, t)$ satisfy $|x_i + 2\om_i t| \leq 2^{n + 1}\delta^{-1}$
for each $i$ and $|t| \leq 2^{n + 1}(2d)^{-1}\delta^{-2} + \delta^{-2} \leq 2^{n + 1}\delta^{-2}$. Therefore the above centered expression is  $\lsm_{q, r} \delta^{-\frac{d + 1}{r}}(2^{n + 1})^{\frac{d}{r} + \frac{1}{q}}\delta^{-\frac{d}{r} - \frac{2}{q}}$.
This gives that
\begin{align*}
III \lsm_{q, r} \delta^{-\frac{2d + 1}{r} - \frac{2}{q}}\sum_{n: 2^{n} \lsm \delta^{-\vep_0^2}}2^{-Edn}(2^{n + 1})^{\frac{d}{r} + \frac{1}{q}} \lsm \delta^{-\frac{2d + 1}{r} - \frac{2}{q}}
\end{align*}
if we choose $E$ sufficiently large. Combining this with \eqref{ivbd1} gives that
\begin{align*}
(\sum_{\om \in \Xi}\nms{f_{\ta_{\om}}}_{L^{q}_{t}L^{r}_{x}}^{2})^{1/2} \lsm_{q, r} (\# \Xi)^{1/2}\delta^{-\frac{2d + 1}{r} - \frac{2}{q}}.
\end{align*}
Comparing this with \eqref{tuninglowerbdhd} gives
\begin{align*}
\mb{D}_{q, r}(\delta) \gtrsim_{q, r} (\# \Xi)^{1/2}\delta^{\frac{d + 2}{r}} - \delta^{1000d}(\# \Xi)^{-1/2}\delta^{\frac{2d + 1}{r} + \frac{2}{q}} \gtrsim \delta^{-(\frac{d}{2} - \frac{d + 2}{r})}
\end{align*}
for all $\delta$ sufficiently small depending on $d$ which proves \eqref{lowerbddlarge}.
\section{Proof of sufficiency of \eqref{qrregion} for Theorem \ref{main}}

Using trivial bounds, it suffices to prove \eqref{mainbd} just for all
$\delta$ sufficiently small depending on $d$.
\subsection{Reduction to smooth decoupling}\label{red}
It will turn out to be more convenient to control a certain more
smoothed version of decoupling.
Let $\chi$ be a Schwartz function which is equal to $1$ on $[-1, 1]^{d + 1}$
and vanishes outside $[-2, 2]^{d + 1}$. If $A_{\om}$ is the linear transformation such that $A_{\om}\ta_{\om} = [-1, 1]^{d + 1}$,
then we define $\chi_{\ta_{\om}} = \chi \circ A_{\om}$. This function is then
equal to 1 on $\ta_{\om}$ and vanishes outside $2\ta_{\om}$. 

For $\ta \subset \R^{d + 1}$, we now define $\wt{P}_{\ta}f := (\wh{f}\chi_{\ta})^{\vee}$. Let $c := 100\sqrt{d}$. We have chosen $c$
sufficiently large so that if $\om$ and $\om'$ are $c\delta$-separated,
then $4\ta_{\om} \cap 4\ta_{\om'} \subset \ta_{\om, 4\delta} \cap \ta_{\om', 4\delta} = \emptyset$.
Given a $c\delta$-separated
subset $\Xi'$ of $[-1, 1]^{d}$, let $\wt{D}_{q, r}(\delta, \Xi')$ be the best constant such that
\begin{align}\label{smoothdef}
\nms{\sum_{\om \in \Xi'}\wt{P}_{\ta_\om}g}_{L^{q}_{t}L^{r}_{x}(\R^d \times \R)} \leq \wt{D}_{q, r}(\delta, \Xi')(\sum_{\om \in \Xi'}\nms{\wt{P}_{\ta_{\om}}g}_{L^{q}_{t}L^{r}_{x}(\R^d \times \R)}^{2})^{1/2}
\end{align}
for all $g$ with Fourier transform supported in $\bigcup_{\om \in \Xi'}2\ta_{\om}$. 
Finally, we let 
\begin{align}\label{smoothdecoupling}
\wt{D}_{q, r}(\delta) := \sup_{\Xi'}\wt{D}_{q, r}(\delta, \Xi')
\end{align}
where the supremum is taken over all $c\delta$-separated subsets $\Xi'$
of $[-1, 1]^{d}$.
The help the reader better recall the notation,
we have opted to use $\wt{D}$ and $\wt{P}$ to denote
the smoothed versions of $D$ and $P$.

The relationship between the decoupling constants \eqref{sharpdecoupling} and \eqref{smoothdecoupling} is contained in the following two propositions.
\begin{prop}[Smooth decoupling implies rough decoupling]\label{smoothrough}
We have $$D_{q, r}(\delta) \lsm \wt{D}_{q, r}(\delta).$$
\end{prop}
\begin{proof}
It suffices to show that for each $\delta$-separated set $\Xi$, we have
that $D_{q, r}(\delta, \Xi) \lsm \wt{D}_{q, r}(\delta)$.
Partition $\Xi$ into $O(1)$ many subsets $\Xi_{i}$, each of which
is $1000c\delta$-separated.
For each $\om \in \Xi_{i}$, our choice of $c$ at the beginning of Section \ref{red} shows that $\ta_{\om}$ intersects one and only one element
of $\{2\ta_{\om'}: \om' \in \Xi_i\}$ (namely $2\ta_{\om}$) 
since the $\Xi_i$ are $1000c\delta$-separated. Thus for $\om, \om' \in \Xi_i$,
\begin{align*}
\wh{f}\chi_{\ta_{\om}}1_{\ta_{\om'}} = \begin{cases}
\wh{f}1_{\ta_{\om}} & \text{ if } \om = \om'\\
0 & \text{ else.}
\end{cases}
\end{align*}
This implies that for each $\om \in \Xi_i$ we have
\begin{align*}
\wt{P}_{\ta_{\om}}(\sum_{\om' \in \Xi_i}P_{\ta_{\om'}}f) = P_{\ta_{\om}}f.
\end{align*}
Therefore
\begin{align*}
\nms{\sum_{\om \in \Xi}P_{\ta_{\om}}f}_{L^{q}_{t}L^{r}_{x}(\R^d \times \R)} &\lsm \max_{i}\nms{\sum_{\om \in \Xi_i}P_{\ta_{\om}}f}_{L^{q}_{t}L^{r}_{x}(\R^d \times \R)}\\
&= \max_{i}\nms{\sum_{\om \in \Xi_i}\wt{P}_{\om}(\sum_{\om' \in \Xi_i}P_{\ta_{\om'}}f)}_{L^{q}_{t}L^{r}_{x}(\R^d \times \R)}\\
&\leq \max_{i}\wt{D}_{q, r}(\delta, \Xi_i)(\sum_{\om \in \Xi_i}\nms{\wt{P}_{\om}(\sum_{\om' \in \Xi_i}P_{\ta_{\om'}}f)}_{L^{q}_{t}L^{r}_{x}(\R^d \times \R)}^{2})^{1/2}\\
&= \max_{i}\wt{D}_{q, r}(\delta, \Xi_i)(\sum_{\om \in \Xi_i}\nms{P_{\ta_{\om}}f}_{L^{q}_{t}L^{r}_{x}(\R^d \times \R)}^{2})^{1/2}\\
&\leq \wt{D}_{q, r}(\delta)(\sum_{\om \in \Xi}\nms{P_{\ta_{\om}}f}_{L^{q}_{t}L^{r}_{x}(\R^d \times \R)}^{2})^{1/2}
\end{align*}
which completes the proof of the proposition.
\end{proof}

\begin{prop}[Rough decoupling implies smooth decoupling]\label{roughsmooth}
We have
$$\wt{D}_{q, r}(\delta) \leq D_{q, r}(2\delta).$$
\end{prop}
\begin{proof}
Given a $2\delta$-separated set $\Xi$, $D_{q, r}(2\delta, \Xi)$ is the best
constant such that
\begin{align}\label{rseq1}
\nms{\sum_{\om \in \Xi}P_{\ta_{\om, 2\delta}}f}_{L^{q}_{t}L^{r}_{x}(\R^d \times \R)} \leq D_{q, r}(2\delta, \Xi)(\sum_{\om \in \Xi}\nms{P_{\ta_{\om, 2\delta}}f}_{L^{q}_{t}L^{r}_{x}(\R^d \times \R)}^{2})^{1/2}
\end{align}
for all $f$ with Fourier transform supported in $\bigcup_{\om \in \Xi}\ta_{\om, 2\delta}$.

Now let $\Xi'$ be an arbitrary $c\delta$-separated subset of $[-1, 1]^d$.
Note that for $\om, \om' \in \Xi'$,
\begin{align}\label{intchi}
1_{\ta_{\om, 2\delta}}\chi_{\ta_{\om', \delta}} = 
\begin{cases}
\chi_{\ta_{\om, \delta}} & \text{ if } \om = \om'\\
0 & \text{ else.}
\end{cases}
\end{align}
Indeed, if $\om = \om'$, then this follows from that $\supp(\chi_{\ta_{\om, \delta}}) = 2\ta_{\om, \delta} \subset \ta_{\om, 2\delta}$. On the other hand if $\om \neq \om'$, then
since $\Xi'$ is $c\delta$-separated, $d(\om, \om') \geq c\delta$.
This implies that
$1_{\ta_{\om, 2\delta}}\chi_{\ta_{\om', \delta}} = 1_{\ta_{\om, 2\delta}\cap 2\ta_{\om', \delta}}\chi_{\ta_{\om', \delta}}$.
But $\ta_{\om, 2\delta} \cap 2\ta_{\om', \delta} = \emptyset$
since any $(\xi, \eta) \in \ta_{\om, 2\delta} \cap 2\ta_{\om', \delta}$
must satisfy $|\xi_i - \om_i|, |\xi_i - \om_{i}'| \leq 2\delta$
and hence $|\om_i - \om_{i}'| \leq 4\delta$ which violates the
fact that $d(\om, \om') \geq c\delta = 100\sqrt{d}\delta$ by our
choice of $c$ at the beginning of Section \ref{red}.
Equation \eqref{intchi} then implies that for each $\om \in \Xi'$,
\begin{align*}
P_{\ta_{\om, 2\delta}}(\sum_{\om' \in \Xi'}\wt{P}_{\ta_{\om', \delta}}g) = \wt{P}_{\ta_{\om, \delta}}g.
\end{align*}

Since $\Xi'$ is also a $2\delta$-separated set, 
for all $g$ with Fourier transform supported in $\bigcup_{\om \in \Xi'}2\ta_{\om, \delta} \subset \bigcup_{\om \in \Xi'}\ta_{\om, 2\delta}$, we have
\begin{align*}
\nms{\sum_{\om \in \Xi'}\wt{P}_{\ta_{\om, \delta}}g}_{L^{q}_{t}L^{r}_{x}(\R^d \times \R)} &= \nms{\sum_{\om \in \Xi'}P_{\ta_{\om, 2\delta}}(\sum_{\om' \in \Xi'}\wt{P}_{\ta_{\om', \delta}}g)}_{L^{q}_{t}L^{r}_{x}(\R^d \times \R)}\\
&\leq D_{q, r}(2\delta, \Xi')(\sum_{\om \in \Xi'}\nms{P_{\ta_{\om, 2\delta}}(\sum_{\om' \in \Xi'}\wt{P}_{\ta_{\om', \delta}}g)}_{L^{q}_{t}L^{r}_{x}(\R^d \times \R)}^{2})^{1/2}\\
&= D_{q, r}(2\delta, \Xi')(\sum_{\om \in \Xi'}\nms{\wt{P}_{\ta_{\om, \delta}}g}_{L^{q}_{t}L^{r}_{x}(\R^d \times \R)}^{2})^{1/2}
\end{align*}
where in both equalities we have used \eqref{intchi}.
Since the above calculation is true for all $g$
with Fourier transform supported in $\bigcup_{\om \in \Xi'}2\ta_{\om, \delta}$, by \eqref{smoothdef}, it follows
that $\wt{D}_{q, r}(\delta, \Xi') \leq D_{q, r}(2\delta, \Xi') \leq D_{q, r}(2\delta).$
Finally taking the supremum over all $c\delta$-separated sets $\Xi'$ then completes the proof.
\end{proof}

We will use Propositions \ref{smoothrough} and \ref{roughsmooth} in the following ways:
\begin{itemize}
\item\label{1stbullet} By Proposition \ref{smoothrough}, to prove that \eqref{qrregion} is sufficient for \eqref{mainbd}, it suffices to show that if $(q, r)$ satisfy \eqref{qrregion}, then $\wt{D}_{q, r}(\delta) \lsm_{q, r, \vep} \delta^{-\vep}$.
\item Bourgain and Demeter's paraboloid decoupling theorem with Proposition \ref{roughsmooth} implies that
$\wt{D}_{p, p}(\delta) \lsm_{p, \vep} \delta^{-\vep}$ for all $2 \leq p \leq \frac{2(d + 2)}{d}$.
\end{itemize}
With this reduction to smoothed decoupling, we are now
ready to begin proving some basic properties
surrounding the smoothed decoupling constant.

\subsection{The $L^{\infty}_{t}L^{2}_{x}$ estimate}
We first prove the bottom right corner of the shaded region in Figure \ref{mixedfig}. More precisely, we will show that $\wt{D}_{\infty, 2}(\delta) \leq 1$.
This estimate follows from Plancherel. Fix an arbitrary $c \delta$-separated subset $\Xi'$ of $[-1, 1]^{d}$.
For each fixed time $t$, we have
\begin{align*}
\nms{\sum_{\om \in \Xi'}\wt{P}_{\ta_{\om}}f}_{L^{2}_{x}(\R^d)} = \nms{\sum_{\om \in \Xi'}\mc{F}_{x}(\wt{P}_{\ta_{\om}}f)}_{L^{2}_{\xi}(\R^d)}
\end{align*}
where $\mc{F}_{x}(g)(\xi, t) = \int_{\R^d}g(x, t)e^{-2\pi i x \cdot \xi}\, dx$ is the Fourier transform in space only. 
Since $\wt{P}_{\ta_{\om}}f$ has Fourier transform supported in $2\ta_{\om}$, 
$\mc{F}_{x}(\wt{P}_{\ta_{\om}}f)$ is supported in a cube of side length $2\delta$
centered at $\om$. Since $\Xi'$ is a $c\delta$-separated set, the supports
of $\mc{F}_{x}(\wt{P}_{\ta_{\om}}f)$ are disjoint.
Thus the above is
\begin{align*}
= (\sum_{\om \in \Xi'}\nms{\mc{F}_{x}(\wt{P}_{\ta_{\om}}f)}_{L^{2}_{\xi}(\R^d)}^{2})^{1/2} = (\sum_{\om \in \Xi}\nms{\wt{P}_{\ta_{\om}}f}_{L^{2}_{x}(\R^d)}^{2})^{1/2}.
\end{align*}
Taking a supremum over all time $t$ and then a supremum over $\Xi$ then shows that $\wt{D}_{\infty, 2}(\delta) \leq 1$. 

\subsection{Interpolation}
We would now like to interpolate the $L^{\infty}_{t}L^{2}_{x}$ estimate
with the estimate for $L^{2}_{t}L^{2}_{x}$ and $L^{\frac{2(d + 2)}{d}}_{t}L^{\frac{2(d + 2)}{d}}_{x}$
to obtain that
$\wt{D}_{q, r}(\delta) \lsm_{q, r, \vep} \delta^{-\vep}$ for any $(q, r)$ such that $(\frac{1}{r}, \frac{1}{q})$
is contained in the triangle bounded by the Strichartz line $\frac{2}{q} + \frac{d}{r} = \frac{d}{2}$ and the lines $\frac{1}{r} = \frac{1}{2}$ and $\frac{1}{r} = \frac{1}{q}$.
We will first need that the smooth Fourier projection onto
$\ta_{\om}$ is a bounded operator in the mixed norm $L^{q}_{t}L^{r}_{x}(\R^{d} \times \R)$.
\begin{lemma}\label{bdd}
For any $\om$,
\begin{align*}
\nms{\wt{P}_{\ta_\om}f}_{L^{q}_{t}L^{r}_{x}(\R^{d} \times \R)} \lsm \nms{f}_{L^{q}_{t}L^{r}_{x}(\R^{d} \times \R)}.
\end{align*}
\end{lemma}
\begin{proof}
In this proof we will make use of a mixed norm Young's inequality, see for example \cite[p. 319]{BP}.
Recall that by definition, we have $\wt{P}_{\ta_{\om}}f = f \ast \wc{\chi}_{\ta}$ and hence
Young's inequality gives
\begin{align*}
\nms{\wt{P}_{\ta_{\om}}f}_{L^{q}_{t}L^{r}_{x}(\R^d \times \R)} \leq \nms{f}_{L^{q}_{t}L^{r}_{x}(\R^d \times \R)} \nms{\wc{\chi}_{\ta}}_{L^{1}_{t}L^{1}_{x}(\R^d \times \R)}. 
\end{align*}
A simple rescaling argument shows that $\wc{\chi}_{\ta}$ has $L^{1}$ norm $\lsm 1$ which completes the proof of the lemma.
\end{proof}
Now we state the interpolation result we will make use of.
\begin{prop}\label{interpolate}
Suppose for some $\alpha \in [0, 1]$ we have that
\begin{align*}
\frac{1}{q} = \frac{\alpha}{q_1} + \frac{1 - \alpha}{q_2}\quad \text{ and } \quad \frac{1}{r} = \frac{\alpha}{r_1} + \frac{1 - \alpha}{r_2}
\end{align*}
with $(q_i, r_i), (q, r) \in [1, \infty]^2$. Then
\begin{align*}
\wt{D}_{q, r}(\delta) \lsm_{q_1, q_2, r_1, r_2}\wt{D}_{q_1, r_1}(\delta)^{\alpha}\wt{D}_{q_2, r_2}(\delta)^{1 - \alpha}.
\end{align*}
\end{prop}
\begin{proof}
Let $\Xi'$ be a $c\delta$-separated subset of $[-1, 1]^d$.
Let $T$ be a linear operator defined by
\begin{align*}
T(\{f_{j}\}_{j \in \Xi'}) := \sum_{j \in \Xi'}\wt{P}_{\ta_{j}}f_{j}.
\end{align*}
Note that $T(\{f_j\})$ has Fourier transform supported in $\bigcup_{j \in \Xi'}2\ta_{j}$.
Then for each $i = 1, 2$, we have
\begin{align*}
\nms{T(\{f_j\})}_{L^{q}_{t}L^{r}_{x}(\R^d \times \R)} &\leq \wt{D}_{q_i, r_i}(\delta)(\sum_{\om \in \Xi'}\nms{\wt{P}_{\ta_{\om}}(\sum_{j \in \Xi'}\wt{P}_{\ta_j}f_j)}_{L^{q}_{t}L^{r}_{x}(\R^d \times \R)}^{2})^{1/2}\\
&\lsm \wt{D}_{q_i, r_i}(\delta)(\sum_{j \in \Xi'}\nms{f_{j}}_{L^{q}_{t}L^{r}_{x}(\R^{d} \times \R)}^{2})^{1/2}
\end{align*}
where in the last inequality we have used Lemma \ref{bdd} and 
that the choice of $c$ at the beginning of Section \ref{red} ensures that the $\{2 \ta_{j}\}_{j \in \Xi'}$ are disjoint.
Thus $T: l^{2}L^{q_i}_{t}L^{r_i}_{x}(\R^{d} \times \R \times \Sigma) \rightarrow L^{q_{i}}_{t}L^{r_i}_{x}(\R^d \times \R)$ for $i = 1, 2$. Applying a mixed norm Riesz-Thorin
interpolation theorem (see for example \cite[p. 316]{BP}), 
shows that $\wt{D}_{q, r}(\delta, \Xi') \lsm_{q_1, q_2, r_1, r_2} \wt{D}_{q_1, r_1}(\delta)^{\alpha}\wt{D}_{q_2, r_2}(\delta)^{1 - \alpha}$. Taking the supremum over all $\Xi'$ then completes the proof.
\end{proof}
This then gives the following corollary which already proves part of the sufficiency of \eqref{qrregion} for Theorem \ref{main}.
\begin{cor}\label{strichartzline}
If $q$ and $r$ are such that $(\frac{1}{r}, \frac{1}{q})$ lies in the closed triangle bounded by the Strichartz line $\frac{2}{q} + \frac{d}{r} = \frac{d}{2}$ and the lines $\frac{1}{r} = \frac{1}{2}$ and $\frac{1}{r} = \frac{1}{q}$, then
$\wt{D}_{q, r}(\delta) \lsm_{q, r, \vep} \delta^{-\vep}$. 
\end{cor}
\begin{proof}
First, interpolate the $L^{\frac{2(d + 2)}{d}}_{t}L^{\frac{2(d + 2)}{d}}_{x}$ estimate with the $L^{\infty}_{t}L^{2}_{x}$ estimate 
to obtain that the estimate $\wt{D}_{q, r}(\delta) \lsm_{q, r, \vep} \delta^{-\vep}$ for all $(q, r)$ such that $\frac{2}{q} + \frac{d}{r} = \frac{d}{2}$.
Next, interpolate each estimate on this Strichartz line with the $L^{2}_{t}L^{2}_{x}$ estimate to obtain the whole region.
\end{proof}
\subsection{Reverse H\"{o}lder in time}
The rest of the proof of the sufficiency of Theorem \ref{main} is devoted to showing that
if we know $L^{q}_{t}L^{r}_{x}$ decoupling
is small, then so is $L^{\wt{q}}_{t}L^{r}_{x}$ for any $2 \leq \wt{q} < q < \infty$.

Fix a $c \delta$-separated
subset $\Xi'$ of $[-1, 1]^{d}$. For $1 \leq q, r \leq \infty$, let
$\ov{D}_{q, r}(\delta, \Xi')$ be the best constant such that
\begin{align*}
\nms{\sum_{\om \in \Xi'}\wt{P}_{\ta_{\om}}f}_{L^{q}_{t}L^{r}_{x}(\R^d \times [0, \delta^{-2}])} \leq \ov{D}_{q, r}(\delta, \Xi')(\sum_{\om \in \Xi'}\nms{\wt{P}_{\ta_{\om}}f}_{L^{q}_{t}L^{r}_{x}(\R^d \times \phi_{[0, \delta^{-2}]})}^{2})^{1/2}
\end{align*}
whenever $f$ has Fourier transform supported in $\bigcup_{ \om \in \Xi'} 2 \ta_{\om}$. Here we recall that $\phi$ is the particular Schwartz function defined in the notation section in the introduction.
Finally, let $$\ov{D}_{q, r}(\delta) := \sup_{\Xi'}\ov{D}_{q, r}(\delta, \Xi')$$
where the supremum is once again taken over all $c\delta$-separated subsets $\Xi'$ of $[-1, 1]^d$.
We will need the following two results which show that the time localized $\ov{D}_{q, r}(\delta)$
is essentially the same size as $\wt{D}_{q, r}(\delta)$.

\begin{prop}\label{equiv}
If $(q, r) \in [2, \infty] \times [1, \infty]$, then
$$\ov{D}_{q, r}(\delta) \lesssim \wt{D}_{q, r}(2 \delta)$$
and
$$\wt{D}_{q,r}(\delta) \lesssim_{q} \ov{D}_{q, r}(\delta).$$
\end{prop}

\begin{prop}\label{decrease}
If $1 \leq \wt{q} < q \leq \infty$, then $$\ov{D}_{\wt{q}, r}(\delta) \lsm_{\wt{q}, q} \ov{D}_{q, r}(\delta).$$
\end{prop}

Combining these two propositions gives the corollary:
\begin{cor}\label{globaldec}
If $2 \leq \wt{q} < q \leq \infty$, then $$\wt{D}_{\wt{q}, r}(\delta) \lsm_{\wt{q}, q} \wt{D}_{q, r}(2 \delta).$$
\end{cor}
Another way to interpret Corollary \ref{globaldec} in terms of Figure \ref{mixedfig}
is that if we know $\wt{D}_{q, r}(\delta) \lsm_{q, r, \vep} \delta^{-\vep}$ for some point
$(\frac{1}{r}, \frac{1}{q})$ (and $\frac{1}{q} \in (0, \frac{1}{2})$), 
then we know the estimate 
$\wt{D}_{\wt{q}, r}(\delta) \lsm_{q, \wt{q}, r, \vep} \delta^{-\vep}$
for all points $(\frac{1}{r}, \frac{1}{\wt{q}})$ above $(\frac{1}{r}, \frac{1}{q})$.

Applying this to $\wt{D}_{\frac{2(d + 2)}{d}, \frac{2(d + 2)}{d}}(\delta) \lsm_{\vep} \delta^{-\vep}$ then shows that $\wt{D}_{2, \frac{2(d + 2)}{d}}(\delta) \lsm_{\vep} \delta^{-\vep}$.
Interpolating with every point on the line $\frac{1}{q} = \frac{1}{r}$, then shows that
$\wt{D}_{q, r}(\delta) \lsm_{q, r, \vep} \delta^{-\vep}$ for all $(\frac{1}{r}, \frac{1}{q})$
in the closed triangle bounded by the lines $\frac{1}{q} = \frac{1}{2}$, $\frac{1}{r} = \frac{d}{2(d + 2)}$, and $\frac{1}{q} = \frac{1}{r}$.
Combining this observation with Corollary \ref{strichartzline}, then shows that if $(q, r)$
satisfy \eqref{qrregion}, then $\wt{D}_{q, r}(\delta) \lsm_{q, r, \vep} \delta^{-\vep}$,
finishing the proof that the constraints in \eqref{qrregion} are sufficient.
Thus it remains to prove Proposition \ref{equiv} and Proposition \ref{decrease}.

\subsubsection{Proof of Proposition \ref{equiv}}
We first prove that global decoupling implies time localized decoupling.

\begin{lemma}%\label{git}
For $1 \leq q, r \leq \infty$, we have $$\ov{D}_{q, r}(\delta) \lsm D_{q, r}(2\delta) \lsm \wt{D}_{q, r}(2 \delta).$$
\end{lemma}
\begin{proof}
It suffices to just prove the first inequality as the second inequality follows from
Proposition \ref{smoothrough}.
Fix an arbitrary $c \delta$-separated subset $\Xi'$. Since $\phi_{[0, \delta^{-2}]} \geq 1_{[0, \delta^{-2}]}$, we have
\begin{align}\label{giteq1}
\nms{\sum_{\om \in \Xi'} \wt{P}_{\ta_{\om}}f}_{L^{q}_{t}L^{r}_{x}(\R^d \times [0, \delta^{-2}])} \leq \nms{\sum_{\om \in \Xi'}(\wt{P}_{\ta_{\om}}f)\phi_{[0, \delta^{-2}]}}_{L^{q}_{t}L^{r}_{x}(\R^d \times \R)}.
\end{align}
Since $\phi_{[0, \delta^{-2}]}(t)$ is Fourier supported in a $\delta^2$ neighborhood
of the origin, $(\wt{P}_{\ta_{\om}}f)(x, t)\phi_{[0, \delta^{-2}]}(t)$ is Fourier supported in $2\ta_{\om, \delta} + (\{0\} \times [-\delta^2, \delta^2])$ which is contained in the
parallelpiped,
\begin{align}\label{convsupport}
\ta^{\ast}_{\om} := \{(\xi_1, \ldots, \xi_{d}, \eta) \in \R^{d + 1}: |\xi_i - \om_i| \leq 2\delta \text{ for } i = 1, 2, \ldots, d, |\eta - 2\om \cdot (\xi -\om) - |\om|^2| \leq 5 d\delta^2\}.
\end{align}
This in turn is contained in $\ta_{\om, 2\delta}$ and hence
$\sum_{\om \in \Xi'}(\wt{P}_{\ta_{\om}}f)\phi_{[0, \delta^{-2}]}$ has Fourier transform
supported in $\bigcup_{j \in \Xi'}\ta_{j, 2\delta}$. 
Since $\Xi'$ is also a $2\delta$-separated set, applying the definition of $D_{q, r}(2\delta, \Xi')$ from \eqref{rseq1} gives that the right hand side of \eqref{giteq1} is equal to
\begin{align}
\begin{aligned}\label{sharpapp}
\|\sum_{j \in \Xi'}P_{\ta_{j, 2\delta}}(\sum_{\om \in \Xi'}(\wt{P}_{\ta_{\om}}f)&\phi_{[0, \delta^{-2}]})\|_{L^{q}_{t}L^{r}_{x}(\R^d \times \R)}\\
&\leq D_{q, r}(2\delta, \Xi')(\sum_{j \in \Xi'}\nms{P_{\ta_{j, 2\delta}}(\sum_{\om \in \Xi'}(\wt{P}_{\ta_{\om}}f)\phi_{[0, \delta^{-2}]})}_{L^{q}_{t}L^{r}_{x}(\R^d \times \R)}^{2})^{1/2}.
\end{aligned}
\end{align}
But now, given a $\ta_{j, 2\delta}$, from \eqref{convsupport} and that $\Xi'$ is $c\delta$-separated, we have that
\begin{align*}
P_{\ta_{j, 2\delta}}(\sum_{\om \in \Xi'}(\wt{P}_{\ta_{\om}}f)\phi_{[0, \delta^{-2}]}) = (\wt{P}_{\ta_{j}}f)\phi_{[0, \delta^{-2}]}
\end{align*}
since we cannot both have $\ta_{j, 2\delta} \cap \ta^{\ast}_{\om} \neq \emptyset$
and $|j - \om| \geq c\delta$ by our choice of $c$. Inserting the above into \eqref{sharpapp}
then shows that $\ov{D}_{q, r}(\delta, \Xi') \leq D_{q, r}(2\delta, \Xi') \leq D_{q, r}(2\delta)$. Taking the supremum over all $c\delta$-separated $\Xi'$ then completes the
proof.
\end{proof}

Next, we need to prove that time localized decoupling implies global decoupling.
\begin{lemma}%\label{tig}
If $1 \leq r \leq \infty$ and $2 \leq q \leq \infty$, then $$\wt{D}_{q, r}(\delta) \lsm_{q} \ov{D}_{q, r}(\delta).$$
\end{lemma}
\begin{proof}
Fix an arbitrary $c\delta$-separated subset $\Xi'$ of $[-1, 1]^d$. Partition $\R$ into intervals of length $\delta^{-2}$. Denote this collection of intervals by
$\mc{P}_{\delta^{-2}}(\R)$. Since $1_{I} \leq \phi_{I}$, we can write
\begin{align*}
\nms{\sum_{\om \in \Xi'}\wt{P}_{\ta_{\om}}f}_{L^q_t L^r_x (\R^d \times \R)} &= (\sum_{I\in\mc{P}_{\delta^{-2}}(\R)} \int_I (\int_{\R^{d}}|\sum_{\om \in \Xi'}\wt{P}_{\ta_{\om}}f|^r dx)^{q/r}dt)^{1/q}\\
&\leq \ov{D}_{q, r}(\delta) (\sum_{I\in \mc{P}_{\delta^{-2}}(\R)} (\sum_{\om\in \Xi'} \nms{\wt{P}_{\ta_{\om}}f}_{L^q_{t}L^r_x(\R^d\times \phi_I)}^2 )^{q/2})^{1/q}.
\end{align*}
Using that $q \geq 2$, by Minkowski's inequality, we can interchange
the $l^{q}_{I}$ and $l^2_{\om}$ norms. Using the definition of $L^{q}_{t}L^{r}_{x}(\R^d \times \phi_{I})$ and observing that
$\sum_{I \in \mc{P}_{\delta^{-2}}(\R)}\phi_{I}^{q} \leq \sum_{I \in \mc{P}_{\delta^{-2}}(\R)}\phi_{I} \lsm 1$ gives that the above is
\begin{align*}
\lsm_{q} \ov{D}_{q, r}(\delta)(\sum_{\om \in \Xi'}\nms{\wt{P}_{\ta_{\om}}f}_{L^{q}_{t}L^{r}_{x}(\R^d \times \R)}^{2})^{1/2}.
\end{align*}
Therefore $\wt{D}_{q, r}(\delta, \Xi') \lsm_{q} \ov{D}_{q, r}(\delta)$.
Taking the supremum over all $\Xi'$ the completes the proof.
\end{proof}

\subsubsection{Proof of Proposition \ref{decrease}}
Applying H\"{o}lder's inequality in time followed by the definition of $\ov{D}_{q, r}(\delta)$, we have
\begin{align*}
\nms{\sum_{\om \in \Xi'}\wt{P}_{\ta_{\om}}f}_{L^{\wt{q}}_{\#, t}L^{r}_{x}(\R^d \times [0, \delta^{-2}])} &\leq \nms{\sum_{\om \in \Xi'}\wt{P}_{\ta_{\om}}f}_{L^{q}_{\#, t}L^{r}_{x}(\R^d \times [0, \delta^{-2}])}\\
&\leq \ov{D}_{q, r}(\delta)(\sum_{\om \in \Xi'}\nms{\wt{P}_{\ta_{\om}}f}_{L^{q}_{\#, t}L^{r}_{x}(\R^d \times \phi_{[0, \delta^{-2}]})}^{2})^{1/2}.
\end{align*}
To finish the proof of Proposition \ref{decrease}, it now remains to show that
\begin{align}\label{proptarget2}
\nms{\wt{P}_{\ta_{\om}}f}_{L^{q}_{\#, t}L^{r}_{x}(\R^d \times \phi_{[0, \delta^{-2}]})} \lsm_{q, \wt{q}}\nms{\wt{P}_{\ta_{\om}}f}_{L^{\wt{q}}_{\#, t}L^{r}_{x}(\R^d \times \phi_{[0, \delta^{-2}]})}
\end{align}
and then apply the definition of $\ov{D}_{\wt{q}, r}(\delta, \Xi')$
followed by taking the supremum over all $\Xi'$.
The validity of \eqref{proptarget2} is 
essentially because of the uncertainty principle, $\wt{P}_{\ta_{\om}}f$ is Fourier supported on time intervals of length $\delta^2$ and thus 
locally constant on time intervals of length $\delta^{-2}$.
The Fourier support of $\wt{P}_{\ta_{\om}}f(x, t)\phi_{[0, \delta^{-2}]}(t)$ is contained in $\ta_{\om}^{\ast}$ as defined in \eqref{convsupport}.
Let $\chi_{\om}^{\ast}$ be a Schwartz function which is $1$ on $\ta_{\om}^{\ast}$ and vanishes
outside $2\ta_{\om}^{\ast}$. Then
$\wt{P}_{\ta_{\om}}f(x, t)\phi_{[0, \delta^{-2}]}(t) = (\wt{P}_{\ta_{\om}}f\phi_{[0, \delta^{-2}]} \ast \wc{\chi}^{\ast}_{\om})(x, t)$.
Young's inequality gives that the left hand side of \eqref{proptarget2} is
\begin{align*}
\delta^{2/q}\nms{\wt{P}_{\ta_{\om}}f\phi_{[0, \delta^{-2}]} \ast \wc{\chi}^{\ast}_{\om}}_{L^{q}_{t}L^{r}_{x}(\R^d \times \R)} &\leq \delta^{2/q}\nms{\wt{P}_{\ta_{\om}}f\phi_{[0, \delta^{-2}]}}_{L^{\wt{q}}_{t}L^{r}_{x}(\R^d \times \R)}\nms{\wc{\chi}^{\ast}_{\om}}_{L^{s}_{t}L^{1}_{x}(\R^d \times \R)}\\
&= (\delta^{2/\wt{q}}\nms{\wt{P}_{\ta_{\om}}f\phi_{[0, \delta^{-2}]}}_{L^{\wt{q}}_{t}L^{r}_{x}(\R^d \times \R)})(\delta^{2/q - 2/\wt{q}}\nms{\wc{\chi}^{\ast}_{\om}}_{L^{s}_{t}L^{1}_{x}(\R^d \times \R)})
\end{align*}
where $1 + \frac{1}{q} = \frac{1}{s} + \frac{1}{\wt{q}}$. 
It is not hard to see that $|\wc{\chi}_{\om}^{\ast}| \sim \delta^{d + 2}$ on a
$O(\delta^{-1}) \times \cdots \times O(\delta^{-1}) \times O(\delta^{-2})$
tube ``dual" to $\ta_{\om}^{\ast}$, pointing in the direction $(-2\om, 1)$ centered at the origin, and decays
rapidly outside this tube (this for example can be formalized in the same manner as in see Proposition \ref{dualtgen}).
Therefore, we have that
$\delta^{2/q - 2/\wt{q}}\nms{\wc{\chi}^{\ast}_{\om}}_{L^{s}_{t}L^{1}_{x}(\R^d \times \R)} \sim_{q, \wt{q}} \delta^{2/q - 2/\wt{q}}\delta^{d + 2}\delta^{-d - 2/s} \lsm 1$. 
This proves \eqref{proptarget2} and concludes the proof of Proposition \ref{decrease}.
\section{Proof of Proposition \ref{full}}
We now prove Proposition \ref{full}.
The proof comes from interpolation.
To show \eqref{fulleq}, we will actually show that
$\wt{D}_{q, r}(\delta)$ is $\lsm_{\vep}$ the right hand side of \eqref{fulleq} and then appeal to
Proposition \ref{smoothrough}.
To show sharpness up to a $\delta^{-\vep}$ loss,
we will use \eqref{lowerboundsum}
and that $D_{q, r}(\delta) \geq \mb{D}_{q, r}(\delta)$.

\subsection{Case 1: $q \leq r$ and $r \geq \frac{2(d + 2)}{d}$}
Since in Proposition \ref{full}, we assumed
$q, r \geq 2$, both $\frac{d}{r} - \frac{d}{2}$
and $\frac{d}{q} - \frac{d}{2}$ are $\leq 0$.
Since $q \leq r$, $\frac{d}{2} - \frac{d}{r} - \frac{2}{q} \leq \frac{d}{2} - \frac{d}{r} - \frac{2}{r}$.
Finally, since $r \geq \frac{2(d + 2)}{d}$,
$\frac{d}{2} - \frac{d + 2}{r} \geq 0$.
Therefore \eqref{lowerboundsum} implies that
$\mb{D}_{q, r}(\delta) \gtrsim_{q, r} \delta^{-(\frac{d}{2} - \frac{d + 2}{r})}$ in this case.
It now remains to show that $\wt{D}_{q, r}(\delta) \lsm_{q, r, \vep} \delta^{-(\frac{d}{2} - \frac{d + 2}{r}) - \vep}$.

We consider two cases. First consider the subcase when $q \leq \frac{2(d + 2)}{d}$. For each $q \in [2, \frac{2(d + 2)}{d}]$, interpolate the estimate
$\wt{D}_{q, \frac{2(d + 2)}{d}}(\delta) \lsm_{q, \vep} \delta^{-\vep}$ from Theorem \ref{main} with the trivial $L^{q}_{t}L^{\infty}_{x}$ estimate $\wt{D}_{q, \infty}(\delta) \lsm \delta^{-d/2}$ which follows from
the triangle inequality and Cauchy-Schwarz.
Note here we have also made use of the first bullet point on Page \pageref{1stbullet}.
Let $\ta$ be the number such that
\begin{align*}
\frac{1}{r} = \frac{\ta}{\frac{2(d+2)}{d}} + \frac{1 - \ta}{\infty}.
\end{align*}
Interpolation from Proposition \ref{interpolate} gives
\begin{align*}
\wt{D}_{q, r}(\delta) \lsm_{q}\wt{D}_{q, \frac{2(d+2)}{d}}(\delta)^{\ta}\wt{D}_{q, \infty}(\delta)^{1 - \ta} \lsm_{q,r,\vep} (\delta^{-\vep})^\theta (\delta^{-d/2})^{1-\ta},
\end{align*}
and by plugging in $\theta=\frac{2(d+2)}{rd}$
we have
\begin{align*}
\frac{d}{2}(1 - \ta)= \frac{d}{2}(1 - \frac{2(d + 2)}{rd}) = \frac{d}{2} - \frac{d + 2}{r}
\end{align*}
which completes the estimate in this subcase.

Next consider the subcase when $q \geq \frac{2(d + 2)}{d}$.
In this case we interpolate the trivial
bound $\wt{D}_{q, \infty}(\delta) \lsm \delta^{-d/2}$
with the Bourgain-Demeter paraboloid estimate:
$\wt{D}_{q, q}(\delta) \lsm_{q, \vep} \delta^{-(\frac{d}{2} - \frac{d + 2}{q}) - \vep}$ for $\frac{2(d + 2)}{d} \leq q \leq \infty$.
Note here we have also used Proposition \ref{roughsmooth}.
Let $\ta'$ be the number such that
\begin{align*}
\frac{1}{r} = \frac{\ta'}{q} + \frac{1 - \ta'}{\infty}.
\end{align*} 
Interpolation then gives
\begin{align*}
\wt{D}_{q, r}(\delta) \lsm_{q} \wt{D}_{q, q}(\delta)^{\ta'}\wt{D}_{q, \infty}(\delta)^{1 - \ta'} \lsm_{q, r, \vep} (\delta^{-(\frac{d}{2} - \frac{d + 2}{q}) - \vep})^{\ta'}(\delta^{-d/2})^{1 - \ta'}.
\end{align*}
Plugging in $\ta' = q/r$ and observing that
\begin{align*}
(\frac{d}{2} - \frac{d + 2}{q})\ta' + \frac{d}{2}(1 - \ta') = (\frac{d}{2} - \frac{d + 2}{q})\frac{q}{r} + \frac{d}{2}(1 - \frac{q}{r}) = \frac{d}{2} - \frac{d + 2}{r}
\end{align*}
then completes the estimate in this subcase.
This completes the proof of the estimate when $q \leq r$ and $r \geq \frac{2(d + 2)}{d}$.

\subsection{Case 2: $q \geq r$ and $\frac{2}{q} + \frac{d}{r} \leq \frac{d}{2}$}
Since $q \geq r$, we have $\frac{d}{2} - \frac{d}{r} - \frac{2}{q} \geq \frac{d}{2} - \frac{d}{r} - \frac{2}{r}$.
Therefore \eqref{lowerboundsum} implies that
$\mb{D}_{q, r}(\delta) \gtrsim_{q, r} \delta^{-(\frac{d}{2} - \frac{d}{r} - \frac{2}{q})}$
in this case. It now remains to show that
$\wt{D}_{q, r}(\delta) \lsm_{q, r, \vep} \delta^{-(\frac{d}{2} - \frac{d}{r} - \frac{2}{q}) - \vep}$.

Let $s$ be such that $\frac{2}{q} + \frac{d}{s} = \frac{d}{2}$.
Since $q \geq r$, $\wt{D}_{q, s}(\delta)\lsm_{q, \vep} \delta^{-\vep}$ by Theorem \ref{main}.
We interpolate this with $\wt{D}_{q, q}(\delta) \lsm_{q, \vep} \delta^{-(\frac{d}{2} - \frac{d + 2}{q}) - \vep}$.
Let $\ta''$ be the number such that
\begin{align*}
\frac{1}{r} = \frac{\ta''}{q} + \frac{1 - \ta''}{s}.
\end{align*}
Interpolation then gives
\begin{align*}
\wt{D}_{q, r}(\delta) \lsm_{q} \wt{D}_{q, q}(\delta)^{\ta''}\wt{D}_{q, s}(\delta)^{1 - \ta''} \lsm_{q, r, \vep} (\delta^{-(\frac{d}{2} - \frac{d + 2}{q}) - \vep})^{\ta''}(\delta^{-\vep})^{1 - \ta''}.
\end{align*}
Then
\begin{align*}
(\frac{d}{2} - \frac{d + 2}{q})\ta'' = (\frac{2}{q} + \frac{d}{s} - \frac{d + 2}{q})\ta'' = d(\frac{1}{s} - \frac{1}{q})\ta'' = d(\frac{1}{s} - \frac{1}{r}) = \frac{d}{2}  - \frac{d}{r}-\frac{2}{q}
\end{align*}
where in the first and last equalities we have used the definition of $s$, in the third equality we have used the definition of $\ta''$.
This completes the proof of the esitmate when $q \geq r$ and $\frac{2}{q} + \frac{d}{r} \leq \frac{d}{2}$.

%\subsection*{Data Availability and Conflict of Interest Statement}
%Data sharing not applicable to this article as no datasets were generated or analysed during the current study.
%On behalf of all authors, the corresponding author states that there is no conflict of interest.

\bibliographystyle{amsplain}
\bibliography{mixed}

\providecommand{\bysame}{\leavevmode\hbox to3em{\hrulefill}\thinspace}
\providecommand{\MR}{\relax\ifhmode\unskip\space\fi MR }
\providecommand{\MRhref}[2]{%
  \href{http://www.ams.org/mathscinet-getitem?mr=#1}{#2}
}
\providecommand{\href}[2]{#2}
\begin{thebibliography}{10}

\bibitem{Barron}
Alex Barron, \emph{On global-in-time {S}trichartz estimates for the
  semiperiodic {S}chr\"{o}dinger equation}, Anal. PDE \textbf{14} (2021),
  no.~4, 1125--1152.

\bibitem{BP}
Agnes~Ilona Benedek and Rafael Panzone, \emph{The space {$L\sp{p}$}, with mixed
  norm}, Duke Math. J. \textbf{28} (1961), 301--324.

\bibitem{BBGL}
Jonathan Bennett, Neal Bez, Susana Guti\'{e}rrez, and Sanghyuk Lee,
  \emph{Estimates for the kinetic transport equation in hyperbolic {S}obolev
  spaces}, J. Math. Pures Appl. (9) \textbf{114} (2018), 1--28.

\bibitem{BD14}
Jean Bourgain and Ciprian Demeter, \emph{The proof of the {$l^2$} decoupling
  conjecture}, Ann. of Math. (2) \textbf{182} (2015), no.~1, 351--389.

\bibitem{BGT}
Nicolas Burq, Patrick G\'{e}rard, and Nikolay Tzvetkov, \emph{Strichartz
  inequalities and the nonlinear {S}chr\"{o}dinger equation on compact
  manifolds}, Amer. J. Math. \textbf{126} (2004), no.~3, 569--605.

\bibitem{KNS}
Shinya Kinoshita, Shohei Nakamura, and Akansha Sanwal, \emph{{D}ecoupling
  inequality for paraboloid under shell type restriction and its application to
  the periodic {Z}akharov system}, arXiv:2212.01805.

\bibitem{thesis}
Zane~Kun Li, \emph{Decoupling for the parabola and connections to efficient
  congruencing}, Ph.D. thesis, UCLA, 2019, available at
  \url{https://escholarship.org/uc/item/0cz3756c}.

\bibitem{247B}
Terence Tao, \emph{247{B}, {N}otes 2: {D}ecoupling theory}, available at
  \url{https://terrytao.wordpress.com/2020/04/13/247b-notes-2-decoupling-theory/}.

\bibitem{YZ}
Yunfeng Zhang, \emph{{O}n {F}ourier restriction type problems on compact {L}ie
  groups}, arXiv:2005.11451.

\bibitem{MO}
Fan Zheng, \emph{Decoupling in mixed norm spaces}, available at
  \url{https://mathoverflow.net/q/229658}.

\end{thebibliography}

\end{document}